\documentclass{paper}

\usepackage{arxiv}

\usepackage[utf8]{inputenc} 
\usepackage[T1]{fontenc}    
\usepackage{hyperref}       
\usepackage{url}            
\usepackage{booktabs}       
\usepackage{amsfonts}       
\usepackage{nicefrac}       
\usepackage{microtype}      
\usepackage{subcaption}
\usepackage{mathrsfs,bm,siunitx}
\usepackage{comment}
\usepackage{appendix}
\usepackage{lipsum}
\usepackage{graphicx}
\usepackage{amsmath,amsthm}
\numberwithin{equation}{section}
\usepackage{color}
\usepackage{colortbl}
\usepackage{xcolor}
\usepackage{transparent}
\usepackage{todonotes}

\graphicspath{{}}
\usepackage{tikz}
\usetikzlibrary{arrows}
\newtheorem{theorem}{Theorem}[section]

\newtheorem{corollary}{Corollary}[section]
\theoremstyle{definition}
\newtheorem{dfn}{Definition}[section]
\newtheorem{remark}{Remark}[section]
\usepackage{mathtools} 
\usepackage{float} 
\usepackage{multirow}

\usepackage{caption}
\title{Tensor field tomography with attenuation and refraction: adjoint operators for the dynamic case and numerical experiments}

\author{
  Lukas Vierus \\
  Department of Mathematics\\
  Saarland University\\
  Saarbr\"ucken, Germany \\
  \texttt{vierus@num.uni-sb.de} \\
       \And
  Thomas Schuster \\
  Department of Mathematics\\
  Saarland University\\
  Saarbr\"ucken, Germany \\
  \texttt{thomas.schuster@num.uni-sb.de} \\
        \And
  Bernadette Hahn \\
  Department of Mathematics\\
  University of Stuttgart\\
  Stuttgart, Germany \\
  \texttt{bernadette.hahn@imng.uni-stuttgart.de} \\
}

\begin{document}
\maketitle
\begin{abstract}


This article is concerned with tensor field tomography in a fairly general setting, that takes refraction, attenuation and time-dependence of tensor fields into account. The mathematical model is given by attenuated ray transforms of the fields along geodesic curves corresponding to a Riemannian metric that is defined by the index of refraction. The data are given at the boundary tangent bundle of the domain and it is well-known that they can be characterized as boundary data of a transport equation turning tensor field tomography into an inverse source problem. This way the adjoint of the forward mapping can be computed using the integral representation or, equivalently, associated to a dual transport equation. The article offers and proves two different representations for the adjoint mappings both in the dynamic and static case. The numerical implementation is demonstrated and evaluated for static fields using the damped Landweber method with Nesterov acceleration applied to both, the integral and PDE-based formulations. The transport equations are solved using a viscosity approximation. The error analysis reveals that the integral representation significantly outperforms PDE-based methods in terms of computational efficiency while achieving comparable reconstruction accuracy. The impact of noise and deviations from straight-line trajectories are investigated confirming improved accuracy if refraction is taken into account. We conclude that the inclusion of refraction to the forward model pays in spite of increased numerical cost.

\end{abstract}

\keywords{dynamic tensor tomography \and inverse source problem \and transport equation \and Riemannian geometry \and viscosity solutions \and adjoint of dynamic ray transform \and geodesic equation \and Landweber method \and Nesterov acceleration
}

\smallskip




\paragraph*{Dedication:} This article is devoted to our friend, mentor and colleague Prof. Dr. Alfred K. Louis (1949-2024) who passed away far too early.\\

\smallskip


\section{Introduction}

Tensor field tomography (TFT) means the computation of a tensor field from given integral data along geodesic curves of a Riemannian metric, which is known as the ray transform of the field. In this article, we examine TFT in a broad framework for both, static and time-dependent fields within a medium that exhibits absorption and refraction. Thereby, the index of refraction determines the Riemannian metric in the considered domain.

TFT has a wide range of applications, including the reconstruction of velocity fields in liquids and gases. We refer to the seminal book of Sharafutdinov \cite{Sharafutdinov_1994} for an overview in theory and applications. The method was first introduced by Norton \cite{norton1988} in 1988, followed by fundamental contributions to Doppler tomography by Juhlin \cite{juhlin1992}, Gullberg \cite{gullberg}, Schuster \cite{schuster2008}, and Strahlen \cite{strahlen1998}. Louis developed a unified approach to inversion formulae of vector and tensor ray transforms \cite{louis2024unified}. A singular value decomposition for the 2D ray transform of vector fields can be found in \cite{derevtsov2011singular}. Prince \cite{prince1996convolution} employed vector tomography in Magnetic Resonance Imaging (MRI), while Panin et al. \cite{panin2002} applied it to diffusion tensor MRI. Procedures for tomography with limited data are discussed in Sharafutdinov \cite{SHARAFUTDINOV:2007}. Among the applications of TFT are diffraction tomography of deformations \cite{lionheart2015}, polarization tomography of quantum radiation \cite{karassiov2004} and the tomography of tensor fields of stresses of e.g. fiberglass composites \cite{puro2007}. In addition, there are polarization tomography \cite{Sharafutdinov_1994}, plasma diagnosis \cite{balandin} and photoelasticity \cite{puro2016}. Furthermore, novel methods exist, which are especially successful in biology and medicine. These include diffusion MRI tomography, which can be used to study the brain in detail. On the other hand, cross-polarized optical coherent tomography allows for a detailed examination of cells and is used for the diagnosis of cancer \cite{panin2002}. Advanced imaging modalities in biomedical contexts, where changes in tissue properties over time must be accurately captured, are addressed, e.g., in \cite{webb2022introduction}. Applying TFT in such areas may involve coupling geometric transport equations with temporal evolution models, requiring new theoretical and computational approaches.

The ray transform of a tensor field $f$ of rank $m\geq 1$ in a bounded Riemannian manifold $(M,g)$ with $M\subset \mathbb{R}^d$ is given by
\begin{align*}
[\mathcal{I}_\alpha f](p,q) = \int_{\gamma_{pq}} f(x)\cdot \dot{\gamma}_{pq}(x)
\exp\Big(-\int_{\gamma_{xq}}\alpha(t)\mathrm{d}\sigma (t)\Big)\mathrm{d}\sigma (x),
\end{align*}
where $\gamma_{pq}$ is a geodesic curve connecting two points $p,q\in \partial M$ and $\alpha\geq 0$ denotes the absorption coefficient. The inverse problem of TFT can then be formulated as:\\[1ex]
\textbf{IP TFT 1:} Compute $f$ from data $\mathcal{I} f$ that are given on a subset $S\subset (\partial M\times \partial M)$.\\[1ex]
It can be shown, see, e.g., \cite{Sharafutdinov_1994}, that this problem is equivalent to computing the source term $f$ in the transport equation
\begin{align}\label{transport_equation}
\mathcal{H} u (x,\xi) + \alpha u (x,\xi) = f\cdot \xi^m,
\end{align}
Here, $\mathcal{H}$ denotes the geodesic vector field corresponding to the metric $g$, which is explicitly given in Section \ref{Sec:ART}, $\xi\in T_x M$ is a tangent vector in $x$ and $\xi^m= \xi\otimes \cdots \otimes\xi$ is the $m$-fold tensor product of $\xi$. The equivalent formulation of IP TFT 1 reads as:\\[1ex]
\textbf{IP TFT 2:} Compute $f$ from \eqref{transport_equation} with given boundary data $u=\mathcal{I} f$ on
a subset $S\subset (\partial M\times \partial M)$.\\[1ex]
This formulation offers the possibility for a holistic approach to TFT in fairly general settings.

For time-depending tensor fields the ray transform extends to 
\begin{align*}
[\mathcal{I}_{\alpha}^d f](t, x,\xi)=\int_{\tau_{-}(x,\xi)}^{0}\langle f(t+\tau, \gamma_{x,\xi}(\tau)),\dot{\gamma}_{x,\xi}^m (\tau)\rangle \exp\Big(-\int_{\tau}^{0}\alpha(\gamma_{x,\xi}(\sigma),\dot{\gamma}_{x,\xi}(\sigma))\mathrm{d}\sigma\Big) \mathrm{d}\tau,
\end{align*}
and the corresponding transport equation then reads as
\begin{align*}
\frac{\partial u}{\partial t}(t,x,\xi) + \mathcal{H} u (t,x,\xi) + \alpha u (t,x,\xi) = f\cdot \xi^m.
\end{align*}
We refer to \cite{Sharafutdinov_1994, derevtsov2021_angular, derevtsov} regarding an introduction and elementary properties of the dynamic ray transform and a deduction of the corresponding transport equation. Dynamic tensor field tomography means then the computation of time-dependent fields $f$ from time-dependent data $\mathcal{I}_{\alpha}^d f$; the two equivalent formulations of this dynamic inverse problem is done in analogy of IP TFT 1 and IP TFT 2. It is obvious that dynamic TFT is highly underdetermined. 

While significant progress has been made in scalar and vector field tomography, the extension to time-dependent tensor fields including refraction and attenuation remains an open field. Classical reconstruction techniques based on static assumptions often fail to provide accurate results. Recent developments in motion-compensation for scalar imaging and regularization strategies for inverse problems provide a foundation for addressing such challenges. The most efficient way to compensate for the dynamics is to incorporate explicit deformation models \cite{hahn2014,katsevich2010}, i.e. their parameters have to be estimated prior to or within the reconstruction step \cite{hahn2017a,feinler2025}. Alternatively, the unknown motion can be treated and compensated for as inexactness in the forward operator \cite{blanke2020inverse}. A general approach for dynamic inverse problems based on Tikhonov-Phillips regularization was presented by Schmitt and Louis \cite{louis2002theo, louis2002appl}. All-at-once and reduced iterative methods to solve time-dependent parameter identification problems are described and analyzed in \cite{kaltenbacher2017} for a fairly general setting of dynamic inverse problems. Moreover, studies such as \cite{hahn2016} analyze how temporal variations influence the resolution and stability of the reconstruction process. In \cite{polyakovahahn2019,PolyakovaSvetov2024} the dynamic longitudinal, mixed and transverse ray transforms acting on symmetric 2-tensor fields were investigated for affine dynamics, leading to the construction of a singular value decomposition and numerical reconstruction methods of type spectral filtering.

In the context of Riemannian geometry, most results to date address static problems. We refer to \cite{eptaminitakis2025tensor, ilmavirta2024tensor, MONARD:14, paternain2021sharp, udoschuster, sharafutdinov1992, STEFANOV;UHLMANN:04, stefanov2018inverting}, what represents only a small selection. The framework introduced in \cite{stefanov2018inverting} for local inversion of the geodesic ray transform near convex boundaries could potentially be extended to time-dependent settings. Investigations like \cite{derevtsov2000influence} already numerically prove the sensitivity of tensor reconstructions to refractive effects, underscoring the need for robust models that account for both, spatial and temporal variations.

In this work, we provide in Section \ref{Sec:ART} the mathematical tools of Riemannian geometry using the metric $g$ associated with the index of refraction and introduce the attenuated geodesic ray transform for static and dynamic tensor fields. In Section \ref{Sec:adjoints} we deliver two different representations for the adjoint mappings using the integral representations on the one side and the transport equations on the other side. Section 4 presents first numerical experiments for the reconstruction of static fields by using the damped Landweber method; the reconstructions of dynamic fields is much more involved and still subject of current research. We use both representations for the adjoint operators and deliver a comparative numerical study. Our findings are summarized in a concluding section (Section 5).

\section{Attenuated ray transforms on manifolds}
\label{Sec:ART}


Let $M\subset \mathbb{R}^d$, $d=2,3$, be a compact domain with strictly convex $C^1$-boundary $\partial M$. According to Fermat’s principle, waves propagate along trajectories that minimize travel time. This observation allows one to interpret such paths as geodesics of a Riemannian manifold $(M,g)$ with metric tensor $g=n^2 I$, where $n(x)$ denotes the index of refraction. We assume, that
\[n(x) \geq c_n > 0 \quad \text{a.e.}\]
for some fixed constant $c_n>0$ implying that $g$ is uniformly positive. The tangent space at a point $x\in M$ is denoted by $T_x M$.
Let $\gamma: [a, b] \rightarrow \mathbb{R}^m$ be a smooth curve. The travel time along $\gamma$ from $\gamma(a)$ to $\gamma(b)$ is given by
\[
T(\gamma) := \int_a^b n(\gamma(\tau)) \|\dot{\gamma}(\tau)\|_{\text{eucl}} \, \mathrm{d}\tau,
\]
where $\Vert\cdot\Vert_{\text{eucl}}$ denotes the Euclidean norm in $\mathbb{R}^3$. This expression motivates the introduction of a Riemannian metric $g = n^2 I$ on $\mathbb{R}^3$ with corresponding line element
\[\mathrm{d}s^2 = n^2(x) \Vert\mathrm{d}x\Vert^2.\]
We employ Einstein's summation convention throughout the article. For a smooth function $f: \mathbb{R}^3 \rightarrow \mathbb{R}$, the Riemannian gradient is given by
\[
\nabla f = g^{ij} \partial_i f \, \partial_j = n^{-2}(x) \nabla_{\text{eucl}} f,
\]
where $g^{ij}$ are the components of the inverse of $g_{ij}$. For tangent vectors $u, v \in T_x M \cong \mathbb{R}^3$ at a given point $x\in M$, the Riemannian inner product is
$\langle u, v \rangle_x = g_{ij}(x) u_i v_j,$
and the corresponding norm satisfies $\Vert u\Vert_x= n(x) \Vert u\Vert_{\text{eucl}}$. For a better readability we omit for the remainder the subscript $x$, i.e. $\Vert \cdot \Vert$ and $\langle \cdot,\cdot\rangle$ always refer to norm and inner product in $(M,g)$.

A curve $\gamma$ that minimizes the functional $T(\gamma)$ is a geodesic with respect to the metric $g$. Such a curve satisfies the geodesic equation
\begin{align}\label{geodesic_eq}
\ddot{\gamma}^k + \Gamma_{ij}^k(\gamma) \dot{\gamma}^i \dot{\gamma}^j = 0
\end{align}
with initial conditions $\gamma (0)=x$, $\dot{\gamma}(0)=\xi\in T_x M.$ The Christoffel symbols $\Gamma_{ij}^k$ associated with $g$ compute as
\[
\Gamma_{ij}^k(x) = n^{-1}(x) \left( \frac{\partial n}{\partial x_j}(x) \delta_{ik} + \frac{\partial n}{\partial x_i}(x) \delta_{jk} - \frac{\partial n}{\partial x_k}(x) \delta_{ij} \right).
\]
The geodesic equation admits a unique solution $\gamma = \gamma_{x,\xi}$, depending smoothly on $x$ and $\xi$. We refer, e.g., to \cite{udoschuster, Sharafutdinov_1994} for details. Following Fermat's principle, in a medium with smooth refractive index, wave propagation occurs along geodesics that are uniquely determined by the initial point and direction.

We demand that for every given point $x\in M$ and non-zero vector $\xi\in T_x M$ the geodesic $\gamma_{x,\xi}(t)$ cannot be extended further than to a finite interval $[\tau_{-}(x,\xi),\tau_{+}(x,\xi)]$. This means that all geodesics have a finite length, where $\gamma_{x,\xi}(\tau_{\mp}(x,\xi))$ are entry and exit point of a geodesic that is for $\tau=0$ at position $x$ and moves in direction $\xi\in T_x M$, see Figure \ref{geo_plot}.

\begin{figure}[H]
\begin{center}
   \includegraphics[scale=0.7]{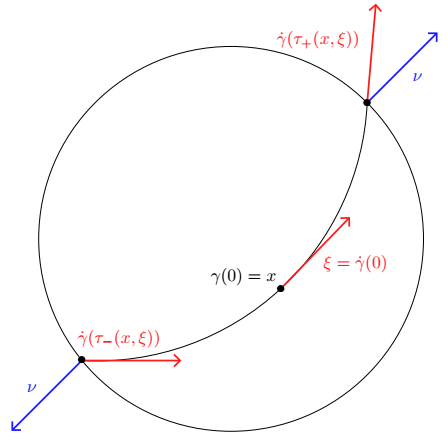}
  \caption{Sketch of a geodesic curve $\gamma_{x,\xi}$}
  \label{geo_plot}
\end{center}
\end{figure}

The tangent bundle of the manifold $M$ is denoted by
\begin{align*}
TM=\lbrace(x,\xi)\vert\ x\in M,\ \xi\in T_x M\rbrace
\end{align*}
and its submanifolds consisting of unit vectors by
\begin{align*}
\Omega_xM := \lbrace \xi\in T_xM \vert \Vert\xi\Vert_x = 1 \rbrace,\qquad
\Omega M= \bigcup_{x\in M} \Omega_x M.
\end{align*}
The boundary of $\Omega M$ can be split into the compact submanifolds
\begin{align*}
\partial_{\pm}\Omega M &= \lbrace (x,\xi)\in \Omega M\vert\ x\in \partial M,\ \pm\langle \xi,\nu(x)\rangle\geq 0\rbrace.
\end{align*}
with $\nu(x)$ being the outer unit normal vector at $x\in \partial M$. Note, that $\tau_\pm$ are smooth on $\partial \Omega M$ (cf. \cite{Sharafutdinov_1994}). Without loss of generality, we assume that $f$ is supported in the unit sphere and set specifically $M:=\lbrace x\in \mathbb{R}^3\colon\ \Vert x\Vert_{\text{eucl}}\leq 1\rbrace$. Given an integer $m\geq 0$, we denote by $S^m$ the space of all functions $\underbrace{\mathbb{R}^3\times\dots\times\mathbb{R}^3}_{m\textit { factors}}\rightarrow \mathbb{R}$ that are $\mathbb{R}$-linear and invariant with respect to all transpositions of the indices. Moreover, we define $\tau_M=(TM, p, M)$ as the tangent bundle and $\tau_{M}'=(T'M, p', M)$ as the cotangent bundle on $M$, where $p:TM\to M$, $p'\colon T'M\to M$ are corresponding projections to $M$, $M'$, respectively. For nonnegative integers $r$ and $s$, we set $\tau_s^r M= (T_s^r M, p_s^r, M)$ as the vector bundle defined by 
\begin{align*}
\tau_r^s M=\underbrace{\tau_{M}\otimes \dots \tau_M}_{r \textrm{ times}} \otimes \underbrace{\tau_{M}'\otimes\dots \otimes \tau_{M}'}_{s \textrm{ times}}.
\end{align*}
We denote the subbundle of $\tau_{m}^{0} M$ consisting of all tensors that are symmetric in all arguments by $S^m \tau_M '$. We have now all ingredients together to define the essential integral transforms.

\begin{dfn}\label{D-ART}
For given $\alpha\in L^{\infty}(\Omega M)$ we define the $\textit{attenuated ray transform}$ of a $m$-tensor field $f=(f_{i_1\cdots i_m})$ by $\mathcal{I}_{\alpha} :L^2(S^m \tau_M ')\rightarrow L^2(\partial_+ \Omega M)$, where
\begin{equation}\label{Iaf}
[\mathcal{I}_{\alpha} f](x,\xi)=\int_{\tau_{-}(x,\xi)}^{0}\langle f(\gamma_{x,\xi}(\tau)),\dot{\gamma}_{x,\xi}^m (\tau)\rangle \exp\left(-\int_{\tau}^{0}\alpha(\gamma_{x,\xi}(\sigma),\dot{\gamma}_{x,\xi}(\sigma))\mathrm{d}\sigma\right) \mathrm{d}\tau.\\
\end{equation}
\end{dfn}

This definition can be extended to time-dependent tensor fields in a straightforward way. 

\begin{dfn}\label{D-DART}
For given $\alpha\in L^{\infty}(\Omega M)$ we define the $\textit{dynamic attenuated ray transform}$ of a $m$-tensor field $f=(f_{i_1\cdots i_m})$ by the mapping $\mathcal{I}_{\alpha}^d f:L^2(0,T; L^2(S^m \tau_M '))\rightarrow L^2(0,T; L^2(\partial_+ \Omega M))$ where
\begin{align}\label{Iadf}
[\mathcal{I}_{\alpha}^d f](t, x,\xi)=\int_{\tau_{-}(x,\xi)}^{0}\langle f(t+\tau, \gamma_{x,\xi}(\tau)),\dot{\gamma}_{x,\xi}^m (\tau)\rangle \exp\left(-\int_{\tau}^{0}\alpha(\gamma_{x,\xi}(\sigma),\dot{\gamma}_{x,\xi}(\sigma))\mathrm{d}\sigma\right) \mathrm{d}\tau.
\end{align}
\end{dfn}

It can be shown that the continuity for static tensor fields as proven in \cite{Sharafutdinov_1994} can be extended to Sobolev-Bochner spaces.

\begin{theorem}Let $(M,g)$ be a Riemannian manifold with $g_{ij}=n^2(x)\delta_{ij}$ and $n\geq 1$. Further, let $\alpha\in H^k(\Omega M)$ with $\alpha\geq \alpha_0 >0$. Then, $\mathcal{I}_\alpha^d$ from \eqref{Iadf} is bounded, i.e., for all $l\in\mathbb{N}_0$ there is a constant $C=C(T,\alpha,\Omega M)>0$, such that
\begin{align}\label{I_bounded}
\Vert \mathcal{I}_\alpha^d f\Vert_{H^l(0,T;H^k(\partial_+\Omega M))}\leq C\Vert f\Vert_{H^l(0,T;H^k(S^m\tau'_M))}.
 \end{align} 
\end{theorem}
\begin{proof}
See \cite{vierus}.    
\end{proof}

As shown in \cite{vierus}, $\mathcal{I}_\alpha^d f$ can be extended to a function $u$ on $[0,T]\times \Omega M$ by
\[
u(t,x,\xi)= \int_{\tau_{-}(x,\xi)}^{0}\langle f(t+\tau, \gamma_{x,\xi}(\tau)),\dot{\gamma}_{x,\xi}^m (\tau)\rangle \exp\left(-\int_{\tau}^{0}\alpha(\gamma_{x,\xi}(\sigma),\dot{\gamma}_{x,\xi}(\sigma))\mathrm{d}\sigma\right) \mathrm{d}\tau,
\]
which solves the boundary value problem
\begin{align}
\frac{\partial u}{\partial t} +\langle \nabla_x u,\xi\rangle -\Gamma_{ij}^{k}\xi^{i}\xi^{j}\frac{\partial u}{\partial \xi^k}+\alpha u &= \langle f,\xi^m\rangle, & &(x,\xi)\in \Omega M\label{eq:transport-dynamic}\\
    u &= \mathcal{I}_\alpha^d f, & &(x,\xi)\in \partial_+ \Omega M\label{eq:transport-dynamic_1}\\
    u &= 0, & &(x,\xi)\in \partial_- \Omega M
    \label{eq:transport-dynamic_2}\\
    u(0,x,\xi) &= 0, & &(x,\xi)\in \Omega M. \label{eq:transport-dynamic_3}
\end{align}
In \cite{vierus} the authors show that, adding a viscosity term $-\varepsilon \Delta u$, $\varepsilon >0$ to the left-hand side of \eqref{eq:transport-dynamic}, the obtained problem has a unique weak solution if the condition
\begin{equation}\label{eq:vierus_condition}
\sup_{x\in M}\frac{\vert \nabla n(x)\vert}{n(x)}<\alpha_0
\end{equation}
is satisfied. Assumption \eqref{eq:vierus_condition} can be interpreted that $n(x)$ varies only slowly in space. This leads to well-defined operators 
\[\mathcal{S}_\alpha^{\varepsilon,d}\colon L^2(0,T; L^2(S^m \tau_M '))\rightarrow L^2(0,T; L^2(\partial_+ \Omega M)),
\qquad \mathcal{S}_\alpha^{\varepsilon,d} (f) :=
u^\varepsilon \vert_{\partial_+ \Omega M},
\]
where $u^\varepsilon$ is the corresponding (weak) viscosity solution of \eqref{eq:transport-dynamic}-\eqref{eq:transport-dynamic_3}. Numerical experiments have confirmed that the solutions of the perturbed equations converge to a solution of \eqref{eq:transport-dynamic}-\eqref{eq:transport-dynamic_3} as $\varepsilon\to 0$ (cf. \cite{vierus}). Any such solution $u$ of \eqref{eq:transport-dynamic}-\eqref{eq:transport-dynamic_3} is then set to $u:= \mathcal{S_\alpha}^d f$. To numerically solve the inverse problem, we aim to derive representations for the adjoint operators \( (\mathcal{I}_\alpha^d)^* \) and \( (\mathcal{S}_\alpha^{d})^* \) since these are necessary, e.g., for algorithms of filtered backprojection type or the Landweber iteration.


\section{Adjoint operators of the ray transform}
\label{Sec:adjoints}


There exist representations for $\mathcal{I}_\alpha^*$ of the integral transform in a non-Euclidean static setting but only for $\alpha=0$. Using Santaló's formula it is shown in \cite{dairbekov2007boundary} that 
\begin{align}\label{adj_If}
[\mathcal{I}_0^*\psi](x) = \int_{\Omega_x M}\psi^\#(x,\xi)\xi^i \xi^j \langle\nu_x,\xi\rangle\mathrm{d}\sigma_{x}(\xi),
 \end{align}
where $\mathrm{d}\sigma_x$ is the surface measure on $\Omega_x M$ and $\psi^\#(x,\xi)$ is defined as a function that equals $\psi(x,\xi)$ on $\partial_{+}\Omega M$ and that is constant along geodesics. Note that $\eqref{adj_If}$ resembles the backprojection operator as it is used, e.g., in computerized tomography (cf. \cite{natterer_book}). In this section we derive adjoint operators for the ray transforms introduced in Section 2 via transport equations for the attenuated and dynamic case.

\begin{theorem}\label{T:adjoint_dynamic}
Let $f\in L^2(0,T; L^2(S^m \tau_M '))$ and $\phi = \mathcal{S}_\alpha^df$. Assuming that for any $h\in L^2(0,T; L^2(\partial_{+} \Omega M))$ there is a unique solution $w\in H^1(0,T; H^1(\Omega M))$ of the adjoint problem
\begin{align}
-\frac{\partial w}{\partial t}-\langle\nabla w,\xi\rangle+\Gamma_{ij}^{k}\xi^i\xi^j\frac{\partial w}{\partial \xi^k}+\left(\alpha+\Xi_n(x,\xi)\right) w &= 0,\quad &t\in [0,T],(x,\xi)\in \Omega M\label{adj_pde}
 \end{align}
with boundary and terminal conditions
\begin{align}
w(T,x,\xi)&=0,\quad (x,\xi)\in\Omega M\label{adj_ic}\\
w(t,x,\xi)&=k_h(t,x,\xi)\exp\Big(-\int_{0}^{\tau_{+}(x,\xi)}\Big(\alpha+\Xi_n\Big)(\gamma_{x,\xi}(\tilde{\tau}),\dot{\gamma}_{x,\xi}(\tilde{\tau}))\mathrm{d}\tilde{\tau}\Big),\label{adj_bc-}\quad t\in [0,T],\ (x,\xi)\in \partial\Omega M,
 \end{align}
where 
\begin{align}
k_h(t,x,\xi) = \frac{h(t,\gamma_{x,\xi}(\tau_{+}(x,\xi),\dot{\gamma}_{x,\xi}(\tau_{+}(x,\xi))}{\langle \nu_{\gamma_{x,\xi}(\tau_{+}(x,\xi)}, \dot{\gamma}_{x,\xi}(\tau_{+}(x,\xi)\rangle}
\end{align}
and
\begin{align*}
    \Xi_n (x,\xi) := \left\{ \begin{array}{cc}
    \frac{1}{2}n^{-1} (x) \langle \nabla n(x),\xi\rangle, &
    d=2\\
    n^{-1} (x) \langle \nabla n(x),\xi\rangle, &
    d=3
    \end{array} \right.
\end{align*}
Then, the adjoint operator can be computed by
\begin{align*}
(\mathcal{S}_\alpha^d)^* h(t,x) = \int_{\Omega_x M}w(t,x,\xi)\xi^m \mathrm{d}\sigma_x(\xi) \in L^2(0,T; L^2(S^m\tau_M ')).
 \end{align*}
\end{theorem}
\begin{proof}
Multiplying \eqref{eq:transport-dynamic} by a test function $w\in H^1(\Omega M)$ and integrating over $[0, T]\times \Omega M$ lead to
\begin{align}\label{weak_adj_eq}
\int_{0}^{T}\int_{\Omega M}\frac{\partial u}{\partial t}w+\langle\nabla_x u,\xi\rangle w-\Gamma_{ij}^{k}\xi^i\xi^j\frac{\partial u}{\partial \xi^k}+\alpha uw\ \mathrm{d}\Sigma\mathrm{d}t= \int_{0}^{T}\int_{\Omega M}f_{i_1,\dots,i_m}(t,x)\xi^{i_1}\cdots\xi^{i_m}w\mathrm{d}\Sigma\mathrm{d}t.    
 \end{align}
Here, $\mathrm{d}\Sigma := \mathrm{d}x \wedge\mathrm{d}\xi$ denotes the volume measure on $\Omega M$. We furthermore compute
\begin{align}\label{adj1}
\int_{0}^{T}\frac{\partial u}{\partial t}w\mathrm{d}t =u(T)w(T)-\int_{0}^{T}u\frac{\partial w}{\partial t}\mathrm{d}t
 \end{align}
and
\begin{align}\label{adj2}
\int_{\Omega M}\langle\nabla_x u,\xi\rangle w\mathrm{d}\Sigma = \int_{\partial_+ \Omega M}\phi w\langle\xi,\nu_x\rangle\mathrm{d}\sigma - \int_{\Omega M}\langle\nabla_x w,\xi\rangle u\mathrm{d}\Sigma.
 \end{align}
An application of \cite[Prop. 3.1]{vierus} furthermore yields
\begin{align}\label{adj3}
-\int_{\Omega_x M}\Gamma_{ij}^{k}\xi^i\xi^j\frac{\partial u}{\partial \xi^k}w\mathrm{d}\Sigma&=\int_{\Omega_x M}\Gamma_{ij}^{k}\xi^i\xi^j\frac{\partial w}{\partial \xi^k}u + \Xi_n(x,\xi) uw\ \mathrm{d}\Sigma.
\end{align}
Inserting \eqref{adj1}-\eqref{adj3} into \eqref{weak_adj_eq} implies
\begin{align}\label{adj_weak1}
&\int_{0}^{T}\int_{\Omega M}u\Big( -\frac{\partial w}{\partial t}  -\langle\nabla_x w,\xi\rangle +\Gamma_{ij}^{k}\xi^i\xi^j\frac{\partial w}{\partial \xi^k}+(\alpha + \Xi_n(x,\xi)) w\Big)\mathrm{d}\Sigma\\\label{adj_weak2}
&+ \int_{0}^{T}\int_{\partial_{+}\Omega M}\phi\Big(w\langle\nu_x,\xi\rangle\Big)\mathrm{d}\sigma\mathrm{d}t\\\label{adj_weak3}
&+\int_{\Omega M} u(T,x,\xi)w(T,x,\xi)\mathrm{d}\Sigma\\
&= \int_{0}^{T}\int_{\Omega M}\langle f, w\xi^m\rangle\mathrm{d}\Sigma\mathrm{d}t\\
&=\langle f, (\mathcal{S}_\alpha^d)^* h\rangle_{L^2(0,T; L^2(S^m \tau'_M))}.
 \end{align}
Since $w$ solves \eqref{adj_pde}, \eqref{adj_ic}, the integrals \eqref{adj_weak1} and \eqref{adj_weak3} vanish and \eqref{adj_weak2} becomes 
\begin{align*}
\int_{0}^{T}\int_{\Omega M} \phi h\mathrm{d}\Sigma = \langle \mathcal{S}_\alpha^d f, h\rangle_{L^2(0,T; L^2(\Omega M))}.
 \end{align*}
Since $u=0$ on $\partial_{-}\Omega M$, the values of $w$ on  $\partial_{-}\Omega M$ are unknown and have to be determined. To this end we restrict $w$ to a geodesic curve $\gamma_{x,\xi}$ for some $(x,\xi)\in \Omega M$. Then, we get for $\tilde{w}(\tau)\coloneqq w(t+\tau,\gamma_{x,\xi}(\tau),\dot{\gamma}_{x,\xi}(\tau))$ that
\begin{align}\label{adj_ode}
\frac{\mathrm{d}\tilde{w}(\tau)}{\mathrm{d}\tau} = \frac{\partial w}{\partial t}(t+\tau,\gamma_{x,\xi}(\tau),\dot{\gamma}_{x,\xi}(\tau)) + \mathcal{H}\tilde{w}(\tau) = (\alpha+\Xi_n)\tilde{w}(\tau)
 \end{align}
and
\begin{align}\label{adj_ode_ic}
\tilde{w}(\tau_{+}(x,\xi))=\frac{h(t+\tau_{+}(x,\xi),\gamma_{x,\xi}(\tau_{+}(x,\xi)),\dot{\gamma}_{x,\xi}(\tau_{+}(x,\xi)))}{\langle \nu_{\gamma_{x,\xi}(\tau_{+}(x,\xi))},\dot{\gamma}_{x,\xi}(\tau_{+}(x,\xi))\rangle}\eqqcolon k_h(t,x,\xi).
 \end{align}
The first order ODE \eqref{adj_ode} for given terminating condition \eqref{adj_ode_ic} can be solved by separation of variables, and we obtain
\begin{align*}
\tilde{w}(\tau)=k_h(t,x,\xi)\exp\left(-\int_{\tau}^{\tau_{+}(x,\xi)}(\alpha+\Xi_n)(\gamma_{x,\xi}(\tilde{\tau}),\dot{\gamma}_{x,\xi}(\tilde{\tau}))\mathrm{d}\tilde{\tau}\right).
 \end{align*}
Thus,
\begin{align}\label{solution_char}
\tilde{w}(0)=w(t,x,\xi) = k_h(t,x,\xi)\exp\left(-\int_{0}^{\tau_{+}(x,\xi)}(\alpha+\Xi_n)(\gamma_{x,\xi}(\tilde{\tau}),\dot{\gamma}_{x,\xi}(\tilde{\tau}))\mathrm{d}\tilde{\tau}\right)
 \end{align}
on $\Omega M$ and, in particular, on $\partial_{-}\Omega M$. This completes the proof.
\end{proof}

\begin{corollary}\label{C:adjoint_static}
In the case of static tensor fields, i.e., $f\in L^2(S^m \tau_M ')$ and $\phi =\mathcal{S}_\alpha(f)$, the adjoint problem is to find the unique solution $w\in H^1(\Omega M)$ of 
\begin{align}
-\langle\nabla w,\xi\rangle+\Gamma_{ij}^{k}\xi^i\xi^j\frac{\partial w}{\partial \xi^k}+\left(\alpha+n^{-1}(x)\langle\nabla n(x),\xi\rangle\right) w &= 0,\qquad &(x,\xi)\in \Omega M\label{adj_pde_stat}
 \end{align}
with boundary conditions
\begin{align}
w(x,\xi)=k_h(x,\xi)\exp\left(-\int_{0}^{\tau_{+}(x,\xi)}\Big(\alpha+\Xi_n\Big)(\gamma_{x,\xi}(\sigma),\dot{\gamma}_{x,\xi}(\sigma))\mathrm{d}\sigma\right), &\qquad(x,\xi)\in \partial\Omega M\label{adj_bc_stat-}
 \end{align}
where
\begin{align}\label{denominator}
    k_h(x,\xi)=\frac{h(\gamma_{x,\xi}(\tau_{+}(x,\xi)),\dot{\gamma}_{x,\xi}(\tau_{+}(x,\xi)))}{\langle \nu_{\gamma_{x,\xi}(\tau_{+}(x,\xi))}, \dot{\gamma}_{x,\xi}(\tau_{+}(x,\xi))\rangle}
 \end{align}
for any $h\in L^2(\partial\Omega M)$. The adjoint operator is then given by 
\begin{align}\label{S_hx}
    \mathcal{S}_\alpha^* h(x)=\int_{\Omega_x M} w(x,\xi)\xi^m\mathrm{d}\sigma_x(\xi) \in L^2(S^m \tau_M ').
 \end{align}
\end{corollary}

\begin{remark}
a) Note that $w$ is well-definied on the boundary since $h(x,\xi)=0$ for $\langle \nu_x,\xi\rangle = 0$. We note that \eqref{S_hx} coincides with \eqref{adj_If} for $\alpha=0$.\\
b) Approximate solutions of the initial and boundary value problems in Theorem \ref{T:adjoint_dynamic} and Corollary \ref{C:adjoint_static} can be obtained by using a viscosity approach, i.e., adding $-\varepsilon \Delta w$ turning the PDEs to a parabolic, respectively, elliptic problem (cf. \cite{vierus}).
\end{remark}

An alternative representation of the adjoint operator exists, which does not rely on the transport equation. We observe, that by \eqref{solution_char} we can explicitly specify $w$ along the entire geodesic curve $\gamma$ using the integral representation. For the dynamic case this way we derive the following alternative representation.

\begin{corollary}
For a given function $h\in L^2(0,T;L^2(\partial_{+}\Omega M))$ we define the adjoint operator
\begin{align*}
(\mathcal{I}_\alpha^d)^{*}\colon L^2(0,T;L^2(\partial_{+}\Omega M))\rightarrow L^2(0,T; L^2(S^m \tau'_M))
\end{align*}   
of the generalized dynamic attenuated ray transform $\mathcal{I}_\alpha^d$ by
\begin{align}\label{eq:adj_dynamic_int}
[(\mathcal{I}_\alpha^{d})^* h](t,x) = \int_{\Omega_x M}w(t,x,\xi)\xi^m\mathrm{d}\sigma_x(\xi),
 \end{align}
where 
\begin{align}\label{eq:w_dyn}
w(t,x,\xi)=k_h(t,x,\xi)\exp\Big(-\int_{0}^{\tau_{+}(x,\xi)}\Big(\alpha+\Xi_n\Big)(\gamma_{x,\xi}(\tilde{\tau}),\dot{\gamma}_{x,\xi}(\tilde{\tau}))\mathrm{d}\tilde{\tau}\Big),
\quad t\in [0,T],\ (x,\xi)\in \Omega M.
 \end{align}
\end{corollary}

\begin{remark}
We emphasize that, from a computational point of view, there is a deciding difference between the representations $(\mathcal{S}_\alpha^d)^{*}$ and $(\mathcal{I}_\alpha^d)^{*}$, though both involve the same backprojection operator over $\Omega_x M$. After solving the initial boundary value problem \eqref{eq:transport-dynamic}-\eqref{eq:transport-dynamic_3}, the integrand $w(t,x,\xi)$ in the evaluation of $(\mathcal{S}_\alpha^d)^{*}h$ is known for \emph{all} $(x,\xi)\in \Omega M$, whereas for the evaluation of $(\mathcal{I}_\alpha^d)^{*}h$ the geodesic equation \eqref{geodesic_eq} has to be solved to compute $\gamma_{x,\xi}$ for each $(x,\xi)\in \Omega M$ separetely, e.g., by the method of characteristics.
\end{remark}


\section{Numerical experiments}


We perform numerical tests for static vector fields in different settings. In all experiments we set $M:=B_1(0)$ the closed unit disc in $\mathbb{R}^2$, $m=1$, i.e. we consider vector fields, and the attenuation coefficient $\alpha \equiv \alpha_0 \geq 0$ to be constant. In Section \ref{Sec:Euclid} we consider the Euclidean geometry ($n\equiv 1$), whereas Section \ref{Sec:non_Euclid} is concerned with variable $n(x)$.

All subsequent computations were implemented in Matlab. We utilized an Intel(R) Core(TM) processor operating at 3.7 GHz with 64 GB of RAM. In cases, where we parallelized the code, we utilized ten cores.


\subsection{Numerical results for the Euclidean case}
\label{Sec:Euclid}

We evaluate the accuracy of the two representations of the adjoint operator $(\mathcal{I}_{\alpha_0})^*$ and $(\mathcal{S}_{\alpha_0})^*$, for synthetic data and $n\equiv 1$. Subsequently, we test different regularization methods to reconstruct vector fields. Once we find the optimal parameters of the grid and the best type of regularization, we extend the examples to metrics associated with variable $n(x)$.

\subsubsection{Numerical approximation of $\mathcal{I}_{\alpha_0} f$}

First, we compute the line integrals of $\mathcal{I}_{\alpha_0} f$. Because of the radial symmetry of $M$, we choose polar coordinates for all variables. Since we only need to consider points $x\in \partial M$, it is sufficient to parameterize as
\begin{align*}
x_p = \begin{pmatrix}
\cos\mu_p\\\sin\mu_p
\end{pmatrix},\qquad \mu_p = \frac{2\pi p}{P},\qquad p=1,\dots,P.
 \end{align*}
For $\xi$ we write accordingly
\begin{align*}
\xi_q = \begin{pmatrix}
\cos\varphi_q\\\sin\varphi_q
\end{pmatrix},\qquad \varphi_q = \frac{2\pi q}{Q},\qquad q=1,\dots,Q.
 \end{align*}
For $(x,\xi)\in\partial_{+}\Omega M$, the integral limit $\tau_{-}(x,\xi)$ of \eqref{Iaf} can be computed explicitly; we have $\tau_{-}(x,\xi) = -2\langle x,\xi\rangle$. Hence, $\mathcal{I}_{\alpha_{0}}$ becomes
\begin{align*}
[\mathcal{I}_{\alpha_{0}}f](x,\xi) = \int_{-2\langle x,\xi\rangle}^{0}\langle f(x+\tau\xi),\xi\rangle \exp(-2\alpha_{0}\langle x,\xi\rangle)\mathrm{d}\tau.
 \end{align*}
To approximate this integral numerically, we use the trapezoidal sum. To this end we need evaluations of $f$ at points $\tilde{x}$ on the line of integration. We distinguish two cases: The first one is that the point $\tilde{x}$ is located between two concentric grid lines, see Figure \ref{inter_out}. 

\begin{figure}[!h]
\begin{center}
\includegraphics[scale=0.20]{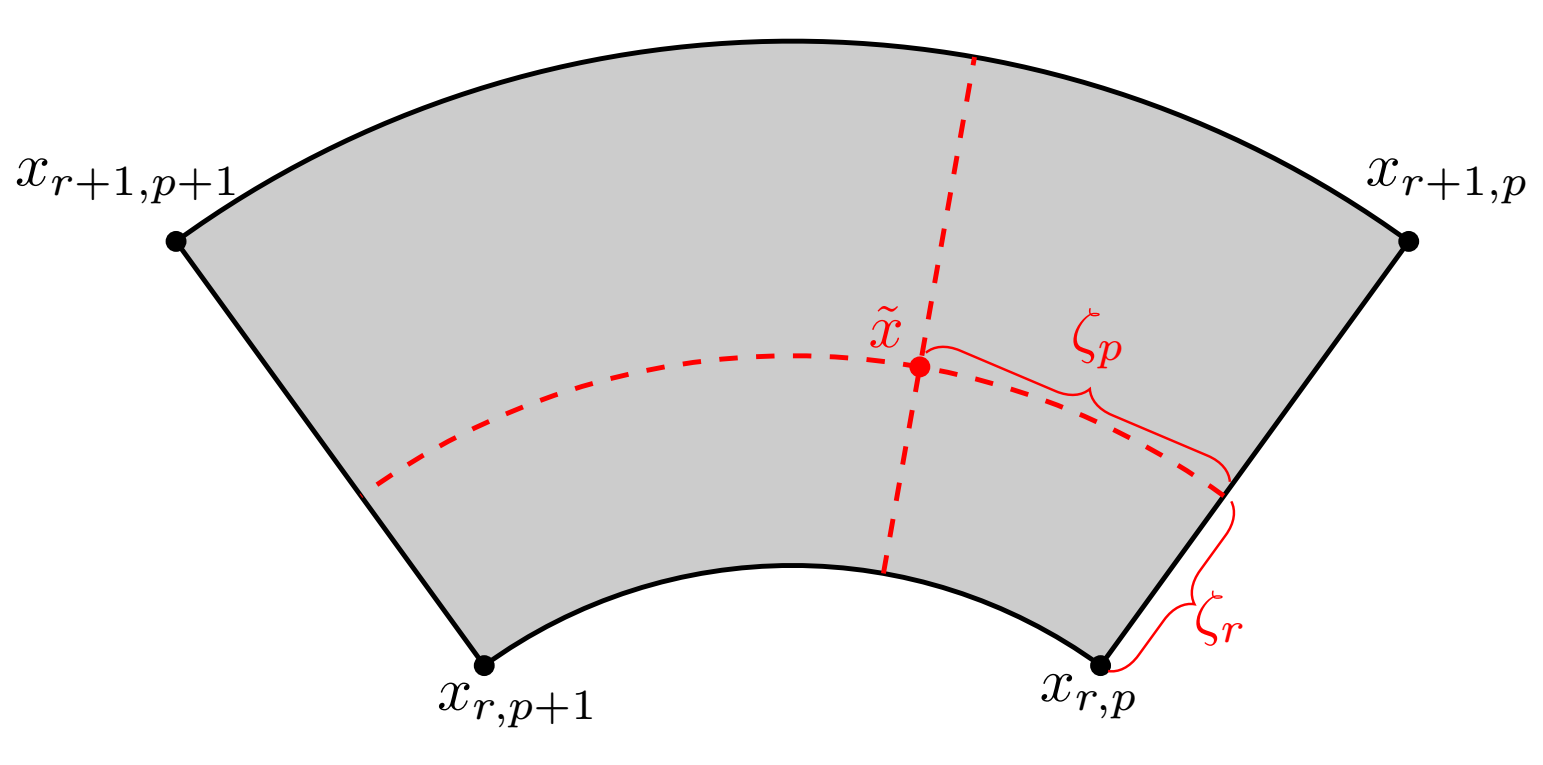}
\caption{Segment between two concentrical circles}
\label{inter_out}
\end{center}
\end{figure}

Without loss of generality $\tilde{x}$ is in the mesh with grid points $x_{r,p}$, $x_{r+1,p}$, $x_{r,p+1}$ and $x_{r+1,p+1}$, see Figure \ref{inter_out}. Let $0\leq\zeta_r \leq 1$ be the radial difference to the inner circle and $0\leq\zeta_p\leq 1$ the angular difference. Then, using bilinear interpolation, we get
\begin{align*}
f(\tilde{x})&\approx \zeta_r\zeta_p f(x_{r,p}) + (1-\zeta_r)\zeta_p f(x_{r+1,p}) \\
&+ \zeta_r(1-\zeta_p) f(x_{r,p+1}) + (1-\zeta_r)(1-\zeta_p) f(x_{r+1,p+1}).
\end{align*}
The second case is, that $\tilde{x}$ lies inside the inner circle. Since $0$ is not contained in the grid, we assign it a value by taking the average of all grid points on the inner circle, i.e.
\begin{align*}
f(0)\coloneqq\frac{1}{P}\sum_{p=1}^{P}f(x_{1,p}).
\end{align*}
For any other $\tilde{x}\neq 0$ there is a segment containing $\tilde{x}$, which is given by two grid points $x_{1,p}$ and $x_{1,p+1}$ and $0$, see Figure \ref{inter_in}. 

\begin{figure}[!h]
\begin{center}
\includegraphics[scale=0.20]{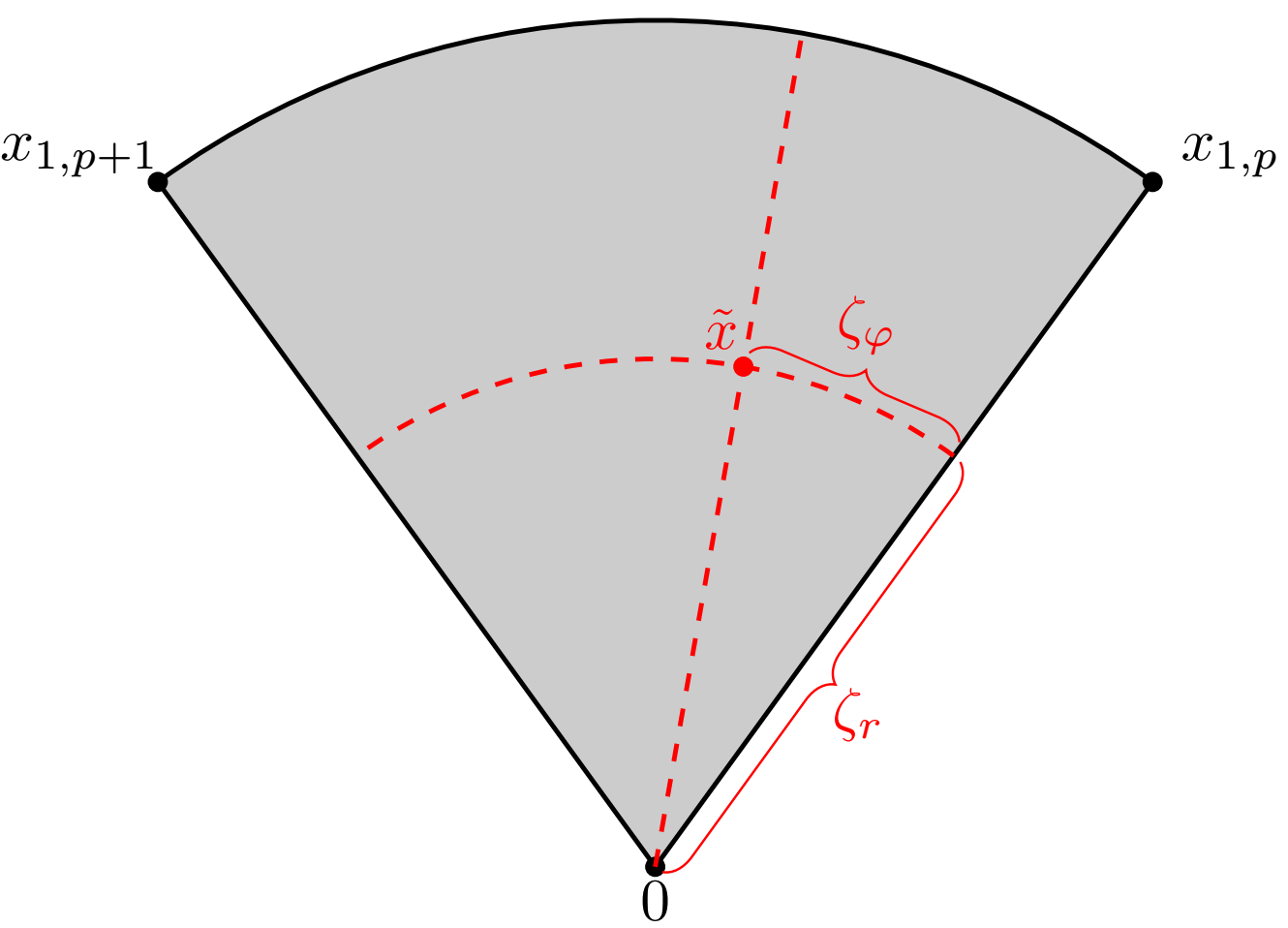}
\caption{Segment inside the inner circle}
\label{inter_in}
\end{center}
\end{figure}

We define $\zeta_r$ and $\zeta_p$ analogously and approximate
\begin{align*}
f(\tilde{x}) \approx \zeta_r f(0) + (1-\zeta_r)(\zeta_p f(x_{1,P})+(1-\zeta_p)f(x_{1,p+1})).
 \end{align*}
Splitting $[\tau_{-}(x,\xi),0]$ into $T$ intervals of length $\Delta\tau = \frac
{-\tau_{-}(x,\xi)}{T}$, we get the step sizes $\tau_t = \tau_{-}(x,\xi)+t\Delta\tau,\ t = 0,\dots,T$, and set
\begin{align*}
H_{p,q} \coloneqq [\mathcal{I}_{\alpha_{0}}f](x_p,\xi_q) = \frac{1}{T+1}\sum_{t=0}^{T}\langle f(x_p+\tau_{t}\xi_q),\xi_q\rangle \exp(\alpha_{0}\tau_t).
 \end{align*}

\subsubsection{Numerical computation of $\mathcal{I}_{\alpha_0}^* h$}

In our setting $\mathcal{I}_{\alpha_0}^* h$ for $h\in L^2(\partial_{+}\Omega M)$ computes as
\begin{align*}
[\mathcal{I}_{\alpha_0}^*h](x) = \int_{0}^{2\pi}w(x,\xi(\varphi))\xi(\varphi)\mathrm{d}\varphi
\end{align*}
where 
\begin{align}
w(x,\xi)&=\frac{h(x+\tau_{+}(x,\xi)\xi,\xi)}{\langle x,\xi\rangle + \tau_{+}(x,\xi)}\exp(-\alpha_{0}\tau_{+}(x,\xi)).\label{w_euk}
\end{align}
For $(x,\xi)\in \Omega M$ we can compute $\tau_{+}(x,\xi)$ explicitly by
\begin{align*}
\tau_{+}(x,\xi)&= -\langle x,\xi\rangle +\sqrt{\langle x,\xi\rangle^2+1-\langle x,x\rangle}.
 \end{align*}
Hence, only the evaluation of $h$ requires interpolation. It is possible to precisely calculate the location where $h$ needs to be evaluated. We parameterize $x_{r,p}\in M$ by
\begin{align*}
x_{r,p} = \rho_r\begin{pmatrix}
\cos\mu_p\\\sin\mu_p
\end{pmatrix},\qquad \rho_r = \frac{r}{R},\qquad r=1,\dots,R
 \end{align*}
and set
\begin{align*}
\tilde{x}_{r,p,q}\coloneqq x_{r,p} +\tau_{+}(x_{r,p},\xi_q)\xi_q\in \partial M
 \end{align*}
for the endpoint of the line of integration. Then, there exists a unique $\tilde{\mu}\in [0,2\pi)$, such that
\begin{align}\label{interpol_euk_x+}
\tilde{x}_{r,p,q}=\begin{pmatrix}
\cos\tilde{\mu}\\\sin\tilde{\mu}
\end{pmatrix}.
\end{align}
Using linear interpolation we approximate
\begin{align*}
h(\tilde{x}_{r,p,q},\xi_q)\approx \left(1-\left(\frac{\tilde{\mu} P}{2\pi} - \left\lfloor \frac{\tilde{\mu} P}{2\pi} \right\rfloor\right)\right) H_{\left\lfloor \frac{\tilde{\mu} P}{2\pi} \right\rfloor,q} + \left(  \frac{\tilde{\mu} P}{2\pi} -\left\lfloor \frac{\tilde{\mu} P}{2\pi} \right\rfloor     \right)H_{\left\lfloor\frac{\tilde{\mu} P}{2\pi}\right\rfloor+1,q}.
 \end{align*}
where the indices are to be considered modulo $P$. The expression $\mathcal{I}_{\alpha_0}^* h$ is then calculated using the trapezoidal sum,
\[
[\mathcal{I}_{\alpha_0}^* h](x) \approx \frac{2\pi}{Q}
\sum_{q=1}^Q w(x,\xi (\varphi_q)) \xi(\varphi_q).
\]

Finally, as reconstruction method we choose the damped Landweber method to iteratively approximate $f$ by
\begin{equation}\label{Landweber}
    f_{k+1} = f_k -\omega \mathcal{I}_{\alpha_0}^* 
    (\mathcal{I}_{\alpha_0}f_k-g^\delta),\qquad k=0,1,\ldots
\end{equation}
with a relaxation parameter $\omega > 0$.

We briefly discuss the ratio between the number of angles $P$ and radii $R$. In Cartesian grids equidistant mesh sizes are usual, since this way the area of the grid meshes is minimized subject to a constant perimeter. Here, we want to proceed similarly: We select the grid sizes in such a way that the straight sides for a grid segment, which are identical for all segments, are approximately as long as the curved sections. We therefore determine the average length of a curved section. The concentric circles in the grid have radii $r_i = \frac{i}{R}, \ i=0,\dots, R$ yielding an average radius $\bar{r}$ of
\[  \bar{r} =\frac{1}{R+1}\sum_{i=0}^R r_i = \frac{1}{2}.\]
Thus, for a grid with $P$ angles, the meshes have an average arc length of $\frac{\pi}{P}$. This should correspond to the lengths of the straight sides of a mesh, i.e. $\frac{\pi}{P} =\frac{1}{R}$ and thus
\begin{align}\label{PpiR}
P\approx\pi R.
\end{align}
We use this as an orientation for our numerical experiments and validate that this ratio, indeed, is reasonable.
Setting $R\cdot P=3600$ fixed, we consider different combinations of $R$ and $P=Q$ and compare the relative errors. As relaxation parameter we use $\omega=0.1$, the initial guess in all experiments is $f_0=0$ and we stop iterating if 
$$\frac{\|f_k - f\|_{L^2}}{\|f\|_{L^2}}< 10^{-5}.$$
Note, that in our experiments the exact solution $f$ is known. The closest pair $(R,P)$ satisfying \eqref{PpiR} with $R\cdot P=3600$ is $(R, P)= (34,106)$. The attenuation coefficients vary as $\alpha=\alpha_0\in\lbrace 0, 0.1,0.2,0.3,0.4\rbrace$.
Table \ref{tab1} shows the reconstruction results for the solenoidal vector field $f^{(1)}(x)=(x_1+x_2,x_1-x_2)^\top$, $(x_1,x_2)\in M$.\\

\begin{table}[H]
\centering
\begin{tabular}{c|ccccc}
\toprule
& & & \boldmath{$\alpha$} & & \\
\textbf{(R, P)} & \textbf{0} & \textbf{0.1} & \textbf{0.2} & \textbf{0.3} & \textbf{0.4} \\
\midrule
(20, 180) & 0.0579 & 0.1432 & 0.4580 & 0.9739 & 1.1155 \\
(30, 120) & 0.0637 & 0.1457 & 0.4599 & 0.9703 & 1.1156 \\
(34, 106) & 0.0233 & 0.1190 & 0.4403 & 0.9568 & 1.1094 \\
(40, 90)  & 0.0313 & 0.1209 & 0.4407 & 0.9567 & 1.1108 \\
(60, 60)  & 0.1190 & 0.1762 & 0.4672 & 0.9540 & 1.1371 \\
(90, 40)  & 1.6989 & 1.7108 & 1.7675 & 1.9603 & 2.2919 \\
(120, 30) & 0.3920 & 0.4228 & 0.6306 & 1.0871 & 1.3950 \\
(180, 20) & 2.9775 & 3.0166 & 3.0895 & 3.2521 & 3.5495 \\
\bottomrule
\end{tabular}
\caption{Relative $L^2$-error of the Landweber approximation  and exact solution $f^{(1)}(x)=(x_1+x_2,x_1-x_2)^\top$ for $R\cdot P=3600$}
\label{tab1}
\end{table}

We observe that, regardless of the choice of $\alpha$, the combination $(R, P) = (34, 106)$ consistently yields the smallest reconstruction errors. Therefore, this combination will be retained for all the remaining experiments. Subsequently, we investigate the optimal number of directions $Q$. We vary $Q$ with $R=34$ and $P=106$ kept constant. The results are illustrated in Table \ref{tab3}.\\

\begin{table}[H]
\centering
\begin{tabular}{c|ccccc}
\toprule
& & & \boldmath{$\alpha$} & & \\
\textbf{Q} & \textbf{0} & \textbf{0.1} & \textbf{0.2} & \textbf{0.3} & \textbf{0.4} \\
\midrule
10  & 0.2330 & 0.2639 & 0.5030 & 0.9684 & 1.1268 \\
20  & 1.1849 & 1.2326 & 1.3388 & 1.6129 & 2.1131 \\
30  & 0.1741 & 0.2146 & 0.4851 & 0.9668 & 1.1090 \\
40  & 0.5509 & 0.5682 & 0.7118 & 1.1090 & 1.5050 \\
50  & 0.1260 & 0.1801 & 0.4745 & 0.9580 & 1.1090 \\
60  & 0.3371 & 0.3585 & 0.5583 & 1.0103 & 1.1270 \\
70  & 0.1514 & 0.1950 & 0.4713 & 0.9671 & 1.1090 \\
80  & 0.1878 & 0.2229 & 0.4839 & 0.9720 & 1.1067 \\
90  & 0.0506 & 0.1268 & 0.4415 & 0.9548 & 1.1056 \\
100 & 0.1858 & 0.2202 & 0.4776 & 0.9714 & 1.1135 \\
106 & 0.0233 & 0.1190 & 0.4403 & 0.9568 & 1.1094 \\
\bottomrule
\end{tabular}
\caption{Relative $L^2$-error of the Landweber approximation and exact solution $f^{(1)}$ for $(R, P) = (34, 106)$ and varying $Q$.}
\label{tab3}
\end{table}
$ $\\[0ex]

In the light of Table \ref{tab3} we obtain the best reconstruction for $Q=106$ independently of $\alpha$. We furthermore observe that for $\alpha = 0.2$ the error for $Q=106$ is comparable to $Q=30$. As the number of arithmetic operations to compute $\mathcal{I}_{\alpha_0}$ and $\mathcal{I}_{\alpha_0}^*$ is proportional to $Q$ we set $P=Q$ for small absorptions $\alpha$ and significantly decrease $Q$ for $\alpha$ large.

 The next experiment reconstructs $f^{(2)}(x)=(x_1^2-2x_2^2,-2x_1x_2)^\top$ for noisy data, see Figure \ref{color}. It can be observed that even for large noise levels $\delta$ the reconstructions prove to be stable. Here, $\delta>0$ means the level of relative uniformly distributed noise, computed in the $L^2$-norm.

\begin{figure}[H]
\begin{center}
   \includegraphics[scale=0.38]{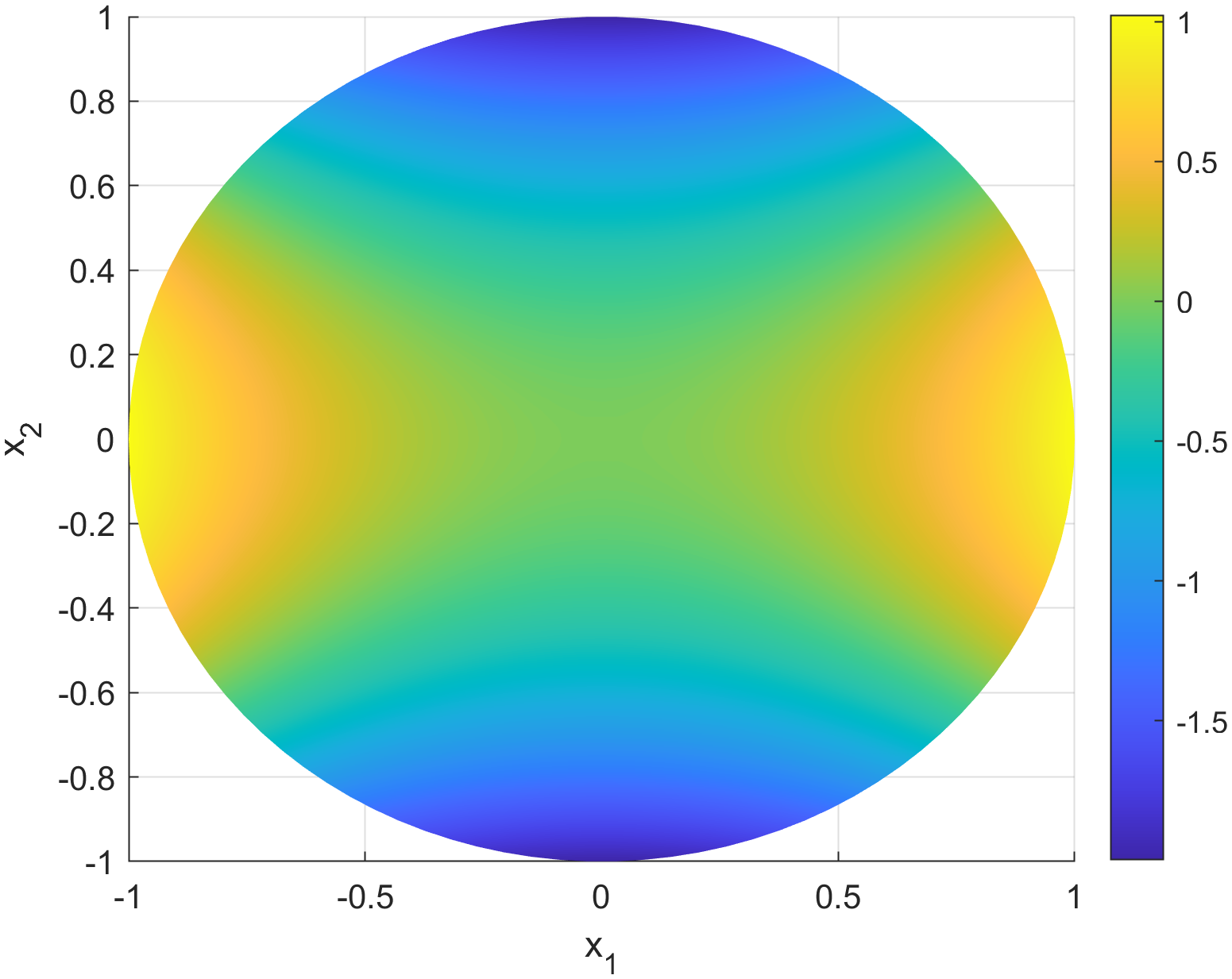}\qquad 
 \includegraphics[scale=0.38]{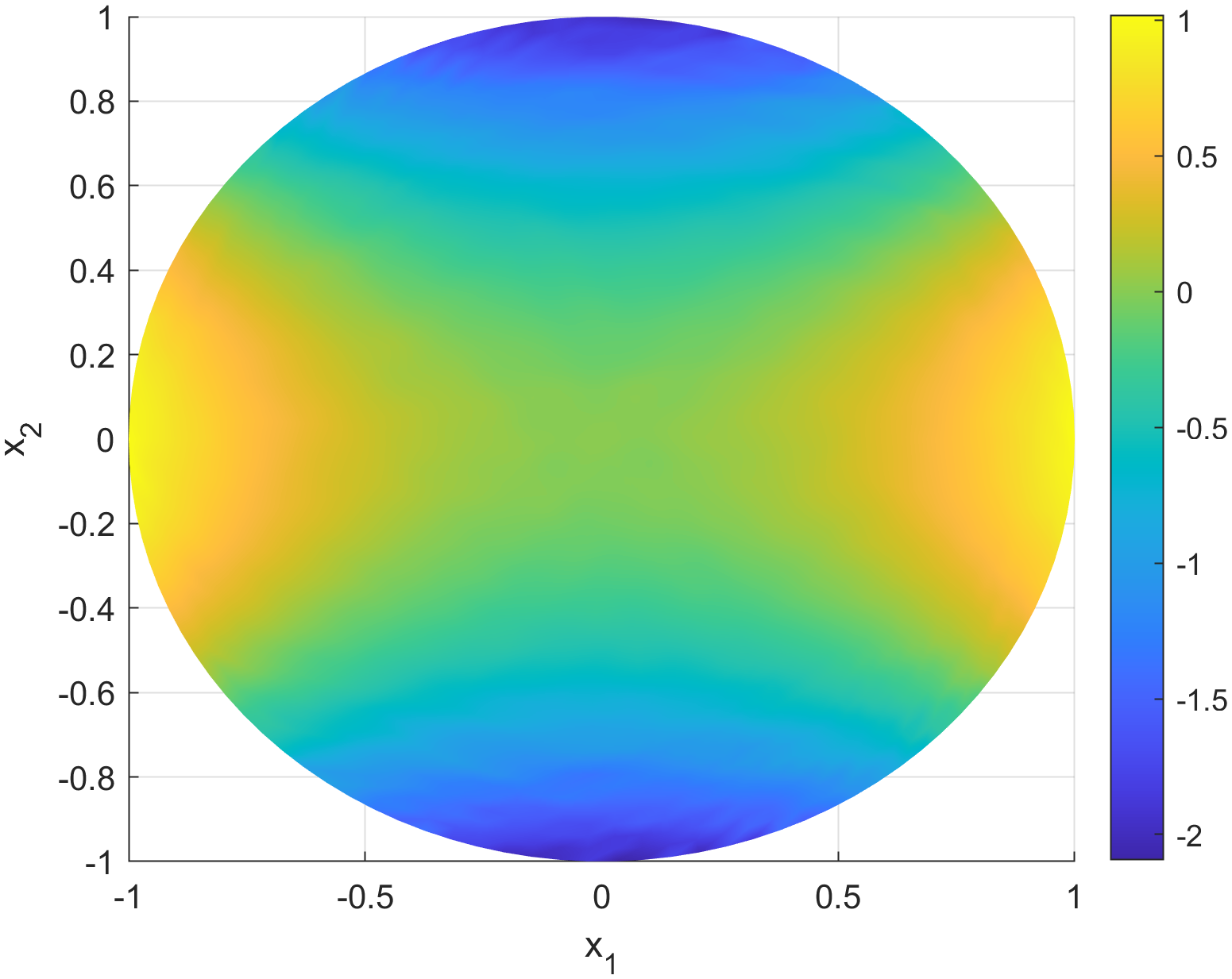}\qquad 
   \includegraphics[scale=0.38]{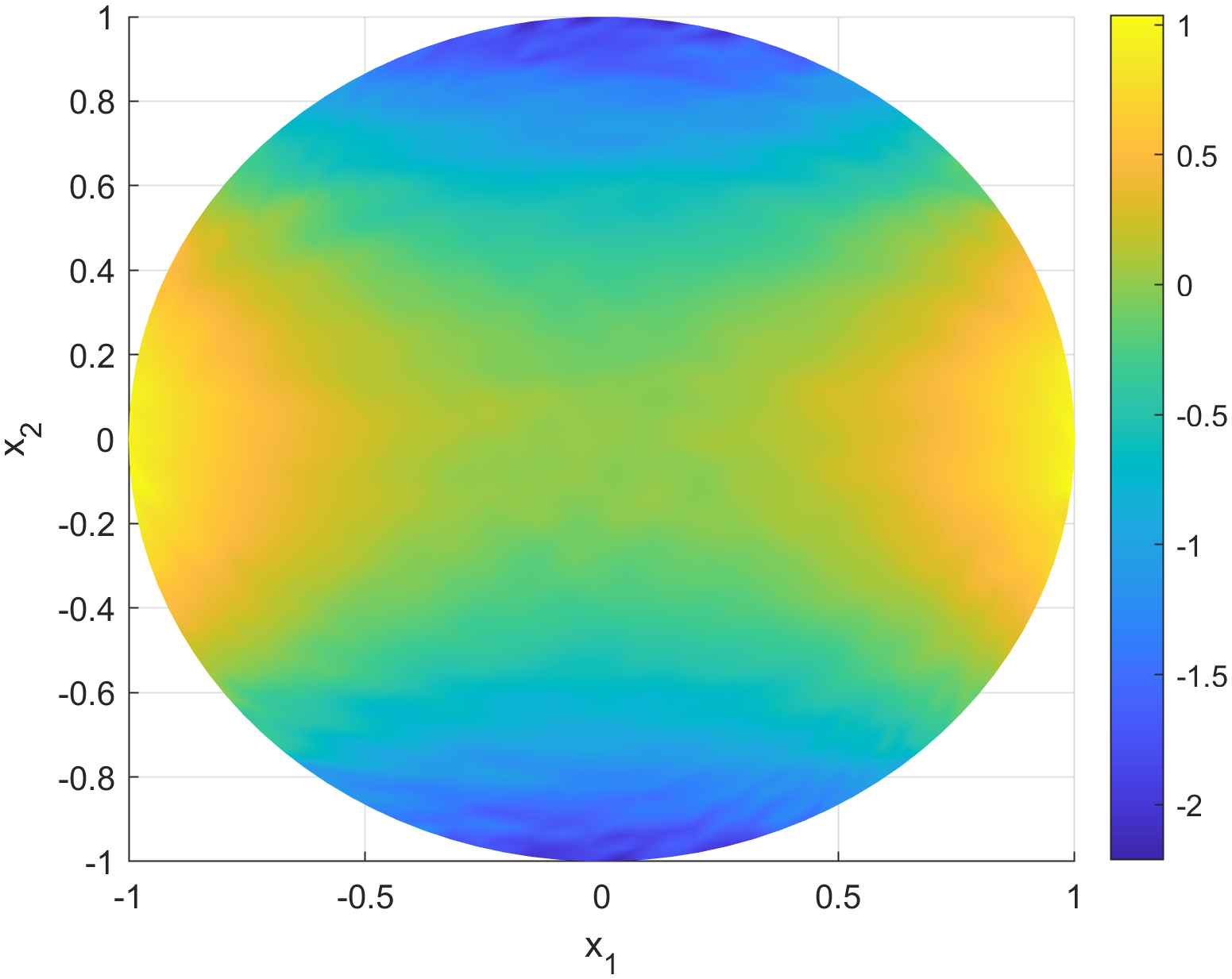}
  \caption{Reconstruction of $f_1^{(2)}$ for varying noise levels $\delta\in\lbrace 0,0.1,0.2\rbrace$} \label{color}
  \end{center}
\end{figure}


\subsubsection{Numerical computation of $\mathcal{S}_{\alpha_0}^* h$}


In this subsection, we examine the alternative approach of computing the adjoint operator by $\mathcal{S}_{\alpha}^*$ instead of $\mathcal{I}_\alpha^*$. 

We recall, that $\mathcal{S}_{\alpha_0}^*\colon  L^2(\partial_{+}\Omega M)\rightarrow L^2(S^1 \tau'_M)$ is given by
\begin{align*}
[\mathcal{S}_{\alpha_0}^*h](x) = \int_{\Omega_x M}w(x,\xi)\xi\mathrm{d}\sigma_x(\xi),
\nonumber\end{align*}
where $w$ solves the following boundary value problem:
\begin{align*}
&-\langle \nabla w(x,\xi),\xi\rangle+\alpha_0 w = 0,& &(x,\xi)\in\Omega M\\[2ex]
w &= \frac{h(x+\tau_{+}(x,\xi)\xi,\xi)}{\sqrt{\langle x,\xi\rangle^2 +1-\vert x\vert^2}},& &(x,\xi)\in \partial\Omega M.
\nonumber\end{align*}
To get a unique weak solution we follow the lines in \cite{vierus} and add $-\varepsilon\Delta w$ for some (small) $\varepsilon>0$ to the left-hand side of the PDE,
\begin{align*}
-\varepsilon\Delta w-\langle \nabla w(x,\xi),\xi\rangle+\alpha_0 w &= 0,& &(x,\xi)\in\Omega M.\end{align*}
We use again spherical coordinates to parameterize $x$ and $\xi$ and calculate the boundary values for $r=R$ analytically via $\mathcal{I}_{\alpha_0}$. For the inner points, we first formulate the PDE in spherical coordinates,
\begin{align}
 -\varepsilon\left(
\frac{1}{\rho}\frac{\partial w}{\partial \rho}+\frac{\partial^2 w}{\partial \rho^2} + \frac{1}{\rho^2}\frac{\partial^2 w}{\partial\mu^2}
\right)
-\frac{\partial w}{\partial \rho}\cos(\varphi-\mu)-\frac{1}{\rho}\frac{\partial w}{\partial\mu}\sin(\varphi-\mu)+\alpha_0 w =0.\label{visco_num}
\nonumber\end{align}
Note, that by a slight misuse of notation we write $w$ instead of $w_\varepsilon$ for the sake of readability. We use finite differences to approximate the partial derivatives. The radial derivative is approximated by forward differences, such that derivatives can also be assigned to points at the innermost ring. The second derivative is approximated using central differential quotients, for the azimuthal derivative we use central differences.
For $r=1$ the second radial derivative cannot be calculated this way. To this end, we assume that $P$ is even. Then, the grid point $x_{1,p+\frac{P}{2}}$ is located on the line that connects $x_{1,p}$ and $x_{2,p}$ at a distance of $2\Delta \rho$ to $x_{1,p}$, see Figure \ref{2nd_derivative}. We choose a linear combination of these three points to approximate the second derivative in $x_{1,p}$. 
A Taylor expansion leads to the approximation
\begin{align*}
\frac{\partial^2 w}{\partial \rho^2}(x_{1,p},\xi_q)
\approx \frac{2w_{2,p,q} -3 w_{1,p,q}+w_{1,p+\frac{P}{2},q}}{3(\Delta\rho)^2}.
\end{align*}
Note that there appears no derivative with respect to $\xi$, as it is constant. This means that the system of equations with $(R-1)PQ$ variables can be reduced to $Q$ systems of equations with a number of $(R-1)P$ variables. These can be calculated independently and, thus, in parallel. 

It remains to investigate the accuracy of the reconstruction if $\mathcal{S}_{\alpha}^*$ is used instead of $\mathcal{I}_{\alpha}^*$. Unfortunately, no rule for the optimal choice of the mesh sizes can be found here. Generally, the method only seems effective for certain grid parameters. Numerous numerical experiments using different combinations of $R$ and $P$ showed that the solution of the adjoint problem is so unstable that no Landweber iteration could be performed. For all grid settings, the matrix arising in each iteration of the finite differences appears to be nearly singular. To avoid this, we calculated viscosity solutions (by adding $-\varepsilon \Delta u$) leading to no improvement. So, we turned over to calculating the minimum norm solution instead. 
The performance is then measured by relative errors
\begin{align}
\textrm{err}(\mathcal{I}_\alpha^*) \coloneqq \frac{\vert \langle \mathcal{I}_\alpha f,\mathcal{I}_\alpha f\rangle - \langle f,\mathcal{I}_\alpha^*\mathcal{I}_\alpha f\rangle\vert}{\langle \mathcal{I}_\alpha f,\mathcal{I}_\alpha f\rangle}
,\qquad
\textrm{err}((\mathcal{S}_\alpha^{\varepsilon})^*) \coloneqq \frac{\vert \langle \mathcal{I}_\alpha f,\mathcal{I}_\alpha f\rangle - \langle f,(\mathcal{S}_\alpha^{\varepsilon})^*\mathcal{I}_\alpha f\rangle\vert}{\langle \mathcal{I}_\alpha f,\mathcal{I}_\alpha f\rangle}.
\end{align}
For various selections of $\varepsilon$, the errors are calculated, the results are depicted in Figure \ref{err_dual2} for $f^{(2)}$. As $\varepsilon$ increases significantly, the error increases, while as $\varepsilon \to 0$, it approaches the error associated with the unperturbed transport equation ($\varepsilon=0$). However, it is noteworthy that in any case the errors are larger than those when using $\mathcal{I}_0^*$. One reason is the fact that we compute the minimum norm solution instead of an exact one. 
 
\begin{figure}[!h]
\begin{center}
\includegraphics[scale=0.20]{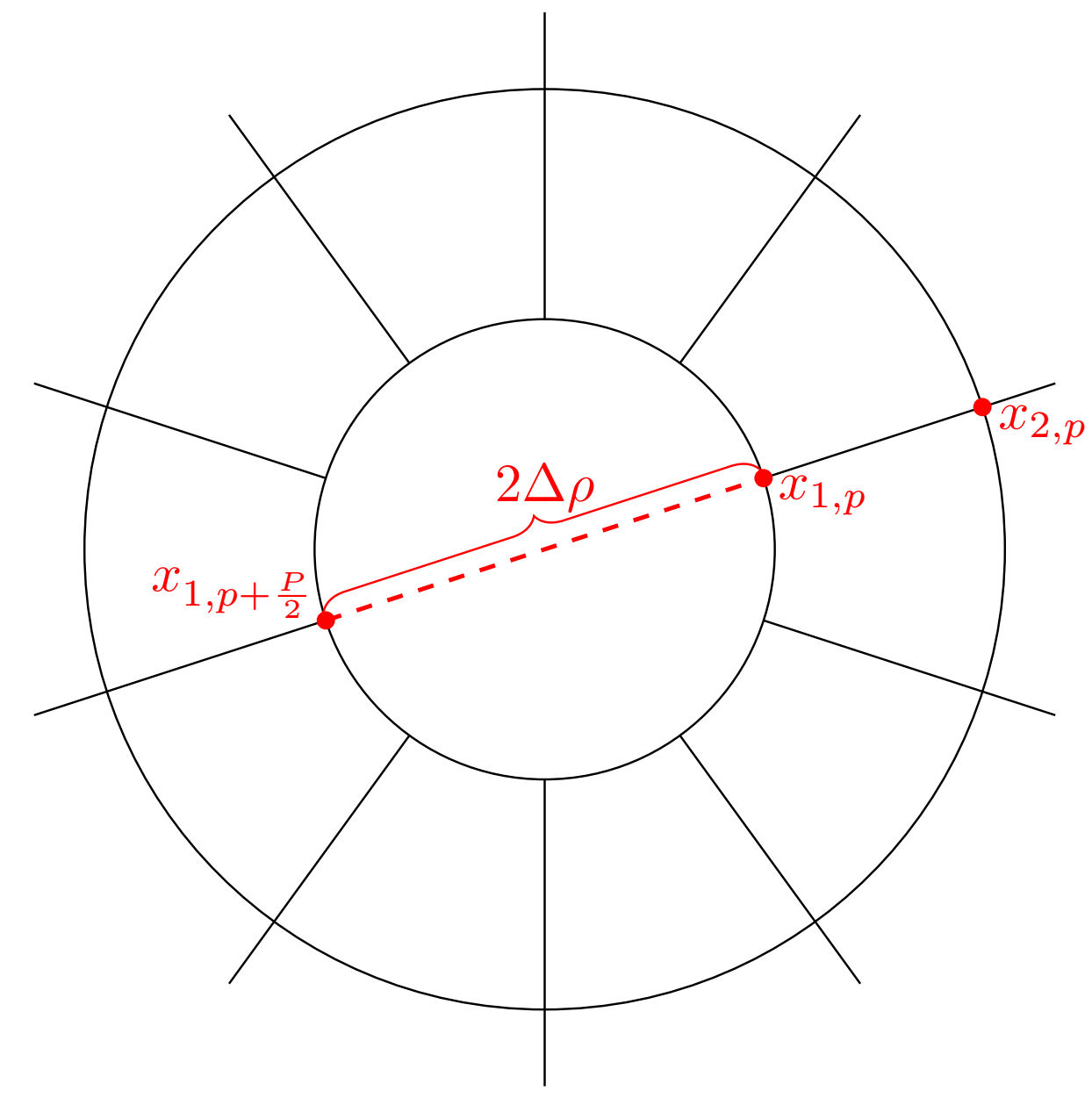}
\caption{Grid points for approximation of $\frac{\partial^2 w}{\partial\mu^2}$ at $p=1$}
\label{2nd_derivative}
\end{center}
\end{figure}
 
Another drawback becomes evident when comparing the run time between the two approaches. When utilizing noise levels of $\delta=0.03$ and $\delta=0.1$ for the same grid and $f_2$, the computing time for $\mathcal{S}_0^*$ is tremendously larger than that for $\mathcal{I}_0^*$, see Table \ref{time}.

\begin{table}[H]
\centering
\begin{tabular}{cccr}
\toprule
Choice of adjoint & Noise level & Relative error & Run time \\
\midrule
\multirow{2}{*}{$\mathcal{I}_0^*$} & 3\%  & 1.30\% & 5 min \phantom{0}5 s \\ 
                                   & 10\% & 3.39\% & 3 min 32 s \\
\midrule
\multirow{2}{*}{$\mathcal{S}_0^*$} & 3\%  & 1.27\% & 178 min 53 s \\ 
                                   & 10\% & 3.34\% & 124 min 25 s \\
\bottomrule
\end{tabular}
\caption{Run time for reconstructing $f_2$ and relative $L^2$-error of Landweber iteration without Nesterov acceleration for both representations of adjoint operators}
\label{time}
\end{table}

At the same time it is remarkable that the reconstruction errors remain almost unchanged, while the run time for iterations using the PDE approach is significantly larger compared to the integral method. Nevertheless the PDE-based expression of the adjoints have to be seen as useful for analytical investigations because of its versatility.

\begin{figure}[H]
\begin{center}
\includegraphics[scale=0.7]{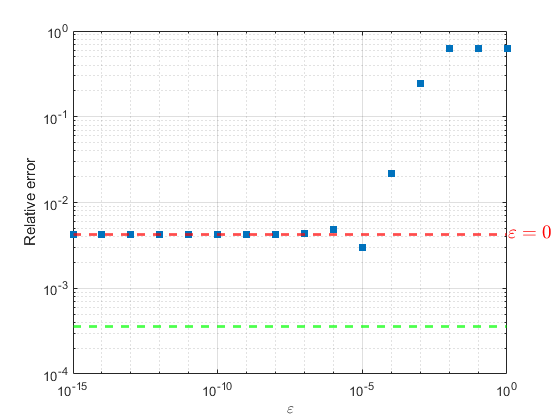}
\caption{Relative errors $\textrm{err}((\mathcal{S}_0^\varepsilon)^*)$ for $\varepsilon> 0$ (blue), $\textrm{err}((\mathcal{S}_0^0)^*)$ (red) and $\textrm{err}(\mathcal{I}_0^*)$ (green) for $f^{(2)}(x)=(x_1^2-2x_2^2,-2x_1 x_2)^\top$}
\label{err_dual2}
\end{center}
\end{figure}

\subsection{Numerical results for variable $n(x)$}
\label{Sec:non_Euclid}

We consider the situation of variable refractive index $n(x)$ and attenuation $\alpha(x)$. Because of the large computation times caused by solving the transport equation we only consider the representation $\mathcal{I}_\alpha^*$ when using the adjoint. Note, that $\xi=\xi(x)$ depends on $x\in M$ for variable $n(x)$ and so do their parameterizations:
\begin{align*}
x_{p}=\begin{pmatrix}
\cos\mu_p\\\sin\mu_p
\end{pmatrix},\quad \mu_p = \frac{2\pi p}{P},\qquad
\xi_{p,q} = n^{-1}(x_{p})\begin{pmatrix}
\cos\varphi_q\\\sin\varphi_q
\end{pmatrix},\quad \varphi_q = \frac{2\pi q}{Q}.
\end{align*}
for $p=1,\ldots,P$, $q=1,\ldots,Q$. First, the geodesic curve $\gamma = \gamma_{x_{p},\xi_{p,q}}$ must be calculated. We restrict to $(x_{p},\xi_{p,q})\in \partial_{+}\Omega M$ and calculate the solution of the geodesic equation \eqref{geodesic_eq} backwards from $\tau=0$ to $\tau = \tau_{-}(x_{p},\xi_{p,q})$. For each pair $(x_{p},\xi_{p,q})$ we denote the grid points by $\tau_s$, $s=0,\dots, S=S_{p,q}$, where $\tau_0=0$
and $\tau_S = \tau_{-}(x_{p},\xi_{p,q})$. The system of ordinary differential equations is solved using the Runge-Kutta scheme of order $4$ and a constant step size $\Delta \tau>0$. We continue calculating points $\gamma(\tau_s)$ of the geodesic until we obtain a point $\gamma(\tau_S)$ in $\mathbb{R}^2\backslash M$. We can assume that $n(x)\approx 1$ at points $x\in \partial M$. (Note, that $n\equiv 1$ in $\mathbb{R}^2\backslash M$.) For sufficiently small $\Delta \tau$, we can therefore assume that $\gamma(\tau_S)$ does not deviate significantly from the exact trajectory. Therefore, we determine the true entry point $\gamma(\tau_{S}^*)$ as the intersection of the line segment connecting $\gamma(\tau_{S-1})$, $\tilde{\gamma}(\tau_{S})$ and the boundary $\partial M$. The discretization of the sampling points to evaluate the integrals is hence no longer constant. We approximate the increment $\Delta\tau^*$ between $\gamma(\tau_{S-1})$ and $\tilde{\gamma}(\tau_{S})$ by scaling
\begin{align*}
\Delta\tau^* = \frac{\Vert\gamma(\tau_{S}) - \gamma(\tau_{S-1})\Vert_{\text{eucl}}}{\Vert {\gamma}(\tau_{S}^*) - \gamma(\tau_{S-1})\Vert_{\text{eucl}}}\cdot \Delta \tau.
\end{align*}
Using the trapezoidal sum, the approximation $A_{p,q,s}$ of the inner integral in \eqref{Iaf} is given for $0<s<S=S_{p,q}$ by
\begin{align*}
-\int_{\tau_s}^{0}\alpha(\gamma_{x_{p},\xi_{p,q}}(\sigma))\mathrm{d}\sigma\approx A_{p,q,s}\coloneqq 
\frac{1}{2} \Delta\tau\left(\sum_{\tilde{s}=1}^{s-1}\alpha(\gamma_{x_{p},\xi_{p,q}}(\tau_{\tilde{s}})) + 
\sum_{\tilde{s}=2}^{s}\alpha(\gamma_{x_{p},\xi_{p,q}}(\tau_{\tilde{s}}))\right).
\end{align*}
For $s=0$ we obviously get $A_{p,q,0}=0$. For $s= S$ we specifically define
\begin{align*}
A_{p,q,S} = A_{p,q,S-1} + \frac{1}{2}\Delta\tau^* \left(\alpha(\gamma_{x_{p},\xi_{p,q}}(\tau_{S-1})) + \alpha(\gamma_{x_{p},\xi_{p,q}}(\tau_{S})) \right).
\end{align*}
This yields
\begin{align*}
\mathcal{I}_{\alpha}f(x_{p},\xi_{p,q})
&\approx - \frac{\Delta\tau}{2}\sum_{s=0}^{S-2}\langle f(\gamma_{x_{p},\xi_{p,q}}(\tau_s)),\dot{\gamma}_{x_{p},\xi_{p,q}}(\tau_s)\rangle \exp( A_{p,q,s})\\
&- \frac{\Delta\tau}{2}\sum_{s=1}^{S-1}\langle f(\gamma_{x_{p},\xi_{p,q}}(\tau_s)),\dot{\gamma}_{x_{p},\xi_{p,q}}(\tau_s)\rangle \exp( A_{p,q,s})\\
&- \frac{\Delta\tau^*}{2} \langle f(\gamma_{x_{p},\xi_{p,q}}(\tau_{S-1})),\dot{\gamma}_{x_{p},\xi_{p,q}}(\tau_{S-1})\rangle \exp( A_{p,q,S-1})\\
&- \frac{\Delta\tau^*}{2} \langle f(\gamma_{x_{p},\xi_{p,q}}(\tau_{S})),\dot{\gamma}_{x_{p},\xi_{p,q}}(\tau_{S})\rangle \exp( A_{p,q,S}).
\end{align*}
Note that all arising inner products are to be understood with respect to the metric $g$. The numerical calculation of $\mathcal{I}_\alpha^*$, which is \eqref{eq:adj_dynamic_int} for static functions $w(x,xi)$, has to be adapted accordingly. For $x\in M$, we parameterize for $r=1,\dots,R$, $p=1,\dots,P$ and $q=1,\dots,Q$
\begin{align*}
x_{r,p}=\rho_r\begin{pmatrix}
\cos\mu_p\\\sin\mu_p
\end{pmatrix},\quad \rho_r = \frac{r}{R} ,\quad \mu_p = \frac{2\pi p}{P},
\qquad
\xi_{r,p,q} = n^{-1}(x_{r,p})\begin{pmatrix}
\cos\varphi_q\\\sin\varphi_q
\end{pmatrix},\quad \varphi_q = \frac{2\pi q}{Q}.
\end{align*}
First, the inner integral in \eqref{eq:adj_dynamic_int} is approximated for variable $\alpha(x)$. To this end, the geodesic equation \eqref{geodesic_eq} is solved forward in time. Again, we calculate further points of the geodesics with a step size $\Delta\tau$ starting from $\tau=0$ until we leave $M$ at $\tau_S$. We determine the true exit point $\gamma(\tau_{S}^*)$ as the intersection of the (Euclidean) line segment of $\gamma(\tau_{S-1})$ to $\tilde{\gamma}(\tau_{S})$ and the boundary $\partial M$. Hence, the last step size $\Delta\tau^*$ is defined as
\begin{align*}
\Delta\tau^* = \frac{\Vert\gamma(\tau_{S}) - \gamma(\tau_{S-1})\Vert_{\mathrm{eucl}}}{\Vert {\gamma}(\tau_{S}^*) - \gamma(\tau_{S-1})\Vert_{\mathrm{eucl}}}\cdot \Delta \tau.
\end{align*}
For $d=2$, we have
\begin{align*}
\Xi_n(x_,\xi) = \frac{1}{2}n^{-1}(x)\langle\nabla n(x),\xi\rangle = \frac{1}{2}n^{-1}(x)\left(\partial_1 n(x)\xi_1 + \partial_2 n(x)\xi_2\right).
\end{align*}
The approximation of \eqref{eq:w_dyn} for $\gamma_{x_{r,p},\xi_{r,p,q}}$ is given by
\begin{align*}
&-\int_{0}^{\tau_{S}}(\alpha+\Xi_n)(\gamma_{x_{r,p},\xi_{r,p,q}}(\sigma),\dot{\gamma}_{x_{r,p},\xi_{r,p,q}}(\sigma))\mathrm{d}\sigma\\
&\approx
- \frac{1}{2} \Delta\tau\left(\sum_{\tilde{s}=1}^{S-2}(\alpha+\Xi_n)(\gamma_{x_{p},\xi_{p,q}}(\tau_s),\dot{\gamma}_{x_{p},\xi_{p,q}}(\tau_s)) + 
\sum_{\tilde{s}=2}^{S-1}(\alpha+\Xi_n)(\gamma_{x_{p},\xi_{p,q}}(\tau_s),\dot{\gamma}_{x_{p},\xi_{p,q}}(\tau_s))\right)\\
&\quad - \frac{1}{2}\Delta\tau^*\left( (\alpha+\Xi_n)(\gamma_{x_{p},\xi_{p,q}}(\tau_{S-1}),\dot{\gamma}_{x_{p},\xi_{p,q}}(\tau_{S-1}))   + (\alpha+\Xi_n)(\gamma_{x_{p},\xi_{p,q}}(\tau_{S}),\dot{\gamma}_{x_{p},\xi_{p,q}}(\tau_{S}))     \right).
\end{align*}
The denominator in \eqref{denominator} is computed as
\begin{align*}
\langle \nu_{\gamma_{x_{r,p},\xi_{r,p,q}}(\tau_S)},\dot{\gamma}_{x_{r,p},\xi_{r,p,q}}(\tau_S)\rangle
= n^{-2}(x_{r,p})\langle\gamma_{x_{r,p},\xi_{r,p,q}}(\tau_S), \dot{\gamma}_{x_{r,p},\xi_{r,p,q}}(\tau_S)\rangle_{\mathrm{eucl}}.
\end{align*}
Finally, we evaluate $h$ in \eqref{denominator} at $(\tilde{x}_{r,p},\tilde{\xi}_{r,p,q})\in\partial_{+}\Omega M$. In contrast to the Euclidean case $n\equiv 1$, the points $\tilde{x}_{r,p}$ and tangents $\tilde{\xi}_{r,p,q}$ are not necessarily located on a grid. Therefore, a two-dimensional interpolation in the variables $\mu$ and $\varphi$ is required. The approximation of the adjoint operator $\mathcal{I}_\alpha^*$ is finally given as
\begin{align*}
[\mathcal{I}_\alpha^* h](x_{r,p})&\approx \frac{2\pi}{Q}\sum_{q=1}^{Q}w(x_{r,p},\xi_{r,p,q})\xi_{r,p,q}.
\end{align*}
We evaluate this approximation by means of several numerical experiments. First, we consider the vector field $f^{(3)}(x) = (x_1,-x_2)^\top$ using $\alpha =0$ and a slowly varying refractive index $n_{0.002}^{(1)}(x) = 0.002\vert x\vert^2 +\frac{4}{3}$ such that condition \eqref{eq:vierus_condition} is satisfied. Again, the Landweber iteration is used as reconstruction method. The absolute errors are visualized in Figure \ref{visco_plot_6}. We see that the reconstructions show a good accuracy even in case of noisy data.

\begin{figure}[H]
\begin{center}
   \includegraphics[scale=0.5]{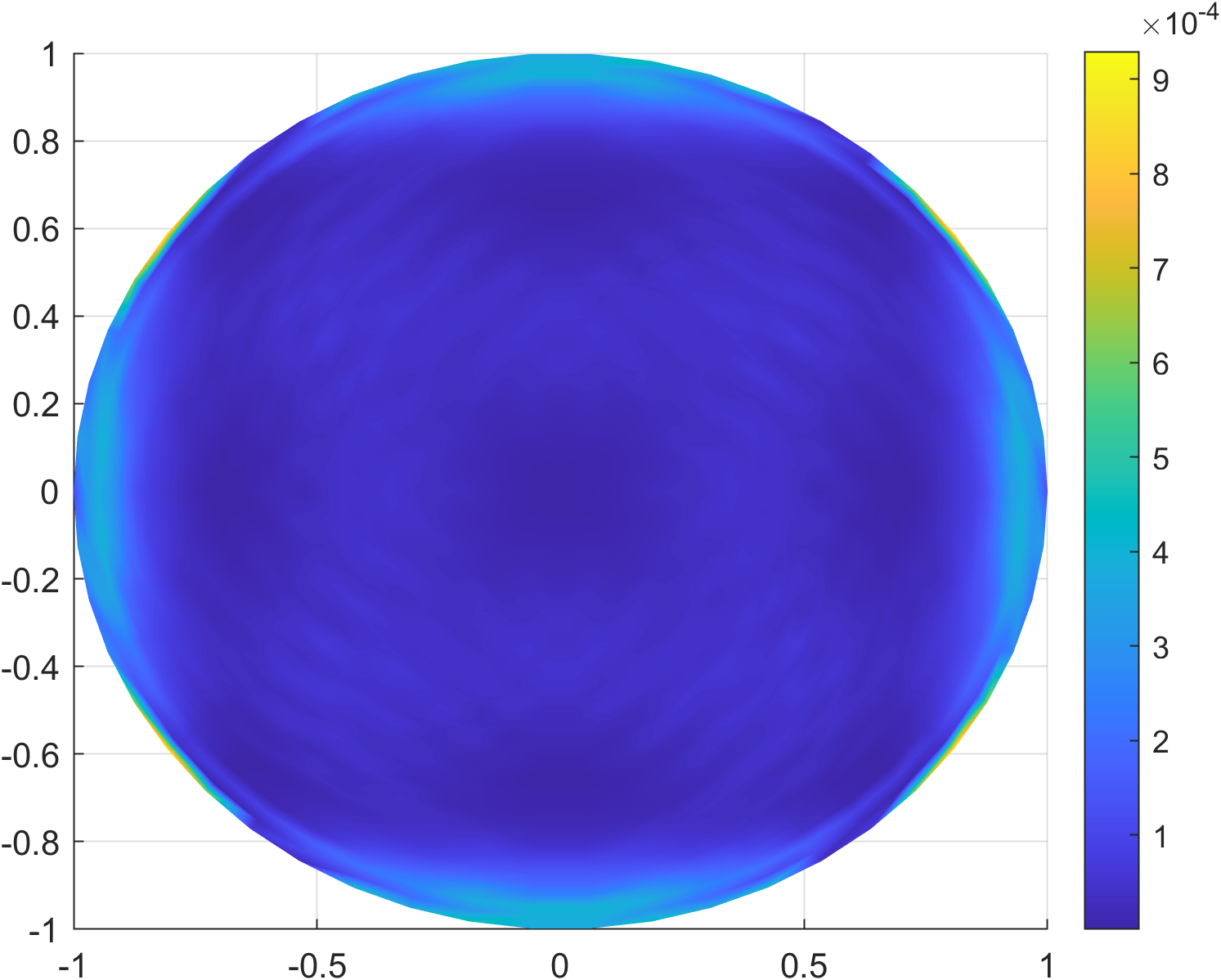}\qquad 
      \includegraphics[scale=0.5]{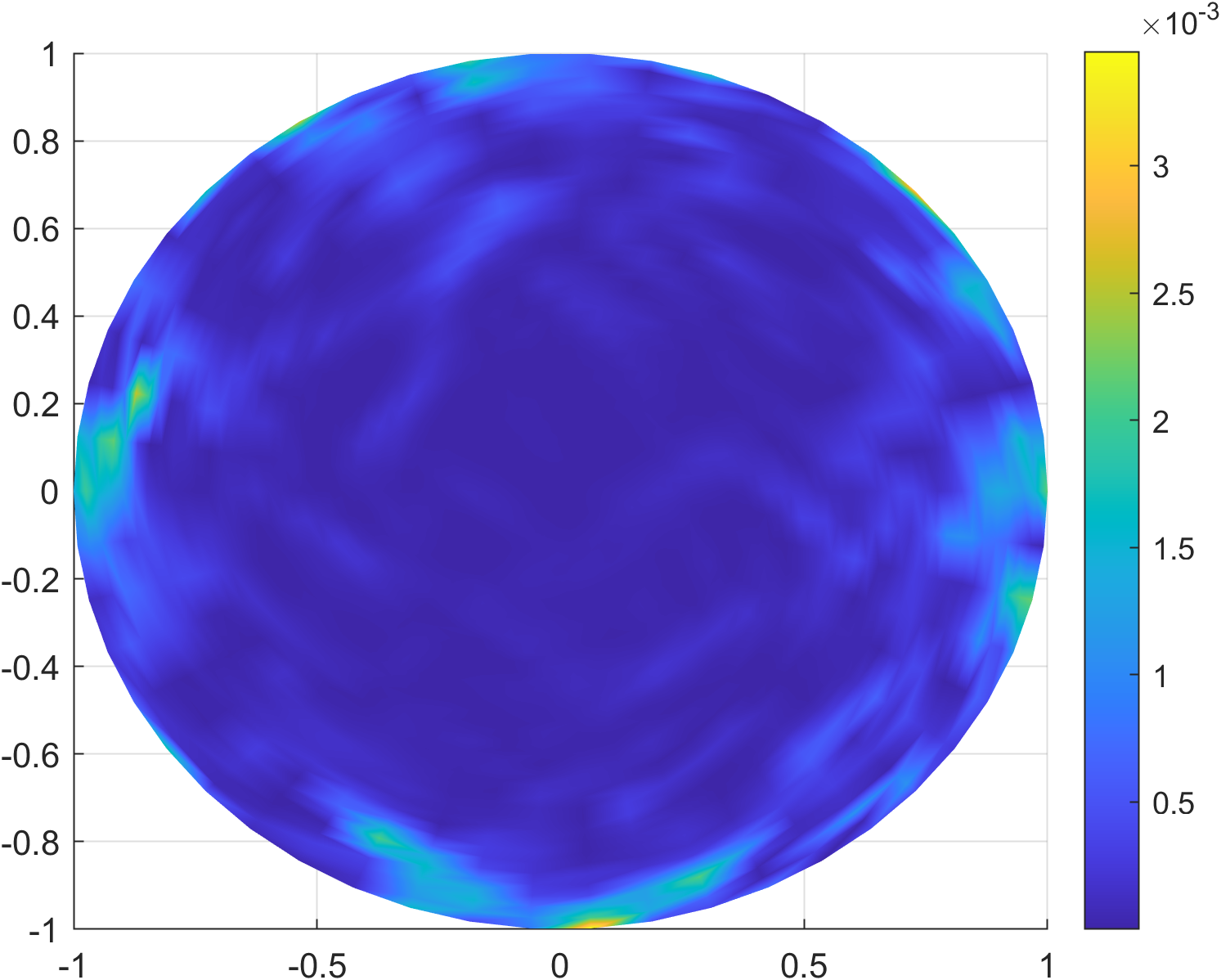}
  \caption{Absolute error in the reconstruction of $f^{(3)}$ for exact data (left picture) and noisy data with $3\%$ uniform noise (right picture) and parameters $(R,P,Q)=(34,106,106)$}
  \label{visco_plot_6}
\end{center}
\end{figure}  

A natural question which arises is, whether incorporating a varying $n(x)$ in the reconstruction really pays regarding both, accuracy and computation time, especially, if the geodesics only slighty deviate from straight lines. For this purpose, we apply $\mathcal{I_\alpha}$ to a vector field, add noise to the result, and then calculate the reconstruction in the Euclidean ($n=1$) setting and including a variable $n(x_1,x_2) = 1+0.002(x_1^2+x_2^2)$. We use Landweber's method with Nesterov's acceleration (cf. \cite{neubauer2017nesterov}) and a relaxation parameter of $\omega = 0.01$. Table \ref{best_result} shows results for different attenuation coefficients $\alpha$.

\begin{table}[H]
\centering
\begin{tabular}{ccccr}
\toprule
Noise level & Attenuation & Refraction (y/n) & Relative $L^2$-error & Run time \\
\midrule
0     & 0.01 & n & 0.0555 &\hfill 21min 21s \\
0     & 0.01 & y & 0.0132 & \hfill 52min \phantom{0}2s \\
0.01  & 0.01 & n & 0.0559 & \hfill 10min 55s  \\
0.01  & 0.01 & y & 0.0144 & \hfill 69min \phantom{0}9s \\
\midrule
0     & 0.02 & n & 0.0556 & \hfill 26min 19s \\
0     & 0.02 & y & 0.0134 & \hfill 145min 34s \\
0.01  & 0.02 & n & 0.0557 & \hfill 14min 40s  \\
0.01  & 0.02 & y & 0.0145 & \hfill 107min \phantom{0}1s \\
\bottomrule
\end{tabular}
  \caption{Comparison of relative error after reconstructing with Euclidean and non-Euclidean model for $(R, P, Q)=(34,106,106)$} \label{best_result}
\end{table}

Table \ref{best_result} clearly shows that, regardless of the choice of noise level or absorption coefficient, the method with refraction taken into account reduces the relative error by about $75\%$, but leading, on the other hand, to a significantly higher run time. Further experiments proved that greater fluctuations in $n(x)$ necessitates the use of the generalized model for variable $n(x)$, since, when ignoring it, the reconstruction errors grow tremendously. 


\section{Conclusions}

Dynamic refractive TFT is a very general framework of inverse problems concerning the reconstruction of tensor fields from integral data along geodesic curves, where the underlying metric is generated by the index of refraction. The inverse problem can be formulated as an inversion of a integral (ray) transform or, equivalently, as inverse source problem for an initial boundary value problem for a transport equation. The adjoint is a backprojection operator, where the integrand is either represented as a integral transform or the solution of a dual boundary value problem with terminating condition. We deduced these representations and performed a numerical comparative study using the Landweber method and different vector fields in two dimensions. Our main findings can be summarized as follows: using the integral representation of the adjoint increases the efficiency of the reconstructions and the incorporation of a variable refractive index is necessary to improve the reconstruction accuracy. Although the PDE representation of the adjoint leads to larger run times, it might be useful for analytical investigations because of its versatility. Future studies comprehend numerical computations taking time-depending tensor fields into account, and further analytical investigations for time-depending fields using both, the integral and PDE representations.


\subsection*{Acknowledgements} The research was funded by German Science Foundation (Deutsche Forschungsgemeinschaft, DFG) under LO 310/17-1 and SCHU 1978/19-1. The third author was further supported by the Sino-German Mobility Programme (M-0187) by the Sino-German Center for Research Promotion.


\bibliography{references} 

@article{balandin,
 author               = {Balandin, AL and Likhachev, AV and Panferov, NV and Pikalov, VV and Rupasov, AA and Shikanov, AS},
 journal              = {Journal of Soviet Laser Research},
 number               = {6},
 pages                = {472--498},
 publisher            = {Kluwer Academic Publishers-Plenum Publishers New York},
 title                = {Tomographic diagnostics of radiating plasma objects},
 volume               = {13},
 year                 = {1992},
 }

@article{blanke2020inverse,
 author               = {Blanke, S.E. and Hahn, B.N. and Wald, A.},
 journal              = {Inverse Problems},
 note                 = {ID 124001},
 number               = {12},
 publisher            = {IOP Publishing},
 title                = {Inverse problems with inexact forward operator: iterative regularization and application in dynamic imaging},
 volume               = {36},
 year                 = {2020},
 }

@article{dairbekov2007boundary,
 author               = {Dairbekov, N.S. and Paternain, G.P. and Stefanov, P. and Uhlmann, G.},
 journal              = {Advances in mathematics},
 number               = {2},
 pages                = {535--609},
 publisher            = {Elsevier},
 title                = {The boundary rigidity problem in the presence of a magnetic field},
 volume               = {216},
 year                 = {2007},
 }

@inproceedings{derevtsov,
 author               = {Derevtsov, Evgeny Yu and Volkov, Yuriy S and Schuster, Thomas},
 booktitle            = {International Conference on Numerical Computations: Theory and Algorithms},
 organization         = {Springer},
 pages                = {97--111},
 title                = {Differential equations and uniqueness theorems for the generalized attenuated ray transforms of tensor fields},
 year                 = {2019},
 }

@article{derevtsov2000influence,
 author               = {E.Y.~Derevtsov and R.~Dietz and A.K.~Louis and T.~Schuster},
 journal              = {Journal of Inverse and Ill-Posed Problems},
 number               = {2},
 pages                = {161--191},
 publisher            = {De Gruyter},
 title                = {Influence of refraction to the accuracy of a solution for the 2{D}-emission tomography problem},
 volume               = {8},
 year                 = {2000},
 }

@article{derevtsov2011singular,
 author               = {Derevtsov, Evgeny Y and Efimov, Anton V and Louis, Alfred K and Schuster, Thomas},
 journal              = {Journal of Inverse and Ill-Posed Problems},
 pages                = {689-716},
 publisher            = {Walter de Gruyter GmbH \& Co. KG},
 title                = {Singular value decomposition and its application to numerical inversion for ray transforms in 2{D} vector tomography},
 volume               = {19},
 year                 = {2011},
 }

@article{derevtsov2021_angular,
 author               = {E.Y.~Derevtsov and Y.S.~Volkov and T.~Schuster},
 journal              = {Applied Mathematics and Computation},
 title                = {Generalized Attenuated Ray Transforms and their Integral Angular Moments},
 volume               = {409},
 year                 = {2021},
 }

@article{eptaminitakis2025tensor,
 author               = {Eptaminitakis, Nikolas and Monard, Fran{\c{c}}ois and Zou, Yuzhou Joey},
 journal              = {arXiv preprint arXiv:2510.04144},
 title                = {Tensor tomography on asymptotically hyperbolic surfaces},
 year                 = {2025},
 }

@article{feinler2025,
 author               = {Feinler, Mathias S and Hahn, Bernadette N},
 journal              = {Sensing and Imaging},
 number               = {4},
 publisher            = {Springer},
 title                = {Fast {G}{A}{N} Based Iterative Motion Estimation for Tomographic Imaging Modalities},
 volume               = {27},
 year                 = {2025},
 }

@unpublished{gullberg,
 author               = {Defrise, M. and Gullberg, G.T.},
 note                 = {https://escholarship.org/uc/item/8df222vs},
 title                = {3{D} reconstruction of tensors and vectors},
 year                 = {2005},
 }

@article{hahn2014,
 author               = {B. Hahn},
 journal              = {Inverse Problems},
 note                 = {ID 035008},
 title                = {Efficient algorithms for linear dynamic inverse problems with known motion},
 volume               = {30},
 year                 = {2014},
 }

@article{hahn2016,
 author               = {Hahn, Bernadette},
 journal              = {Inverse Problems},
 note                 = {ID 025006},
 number               = {2},
 publisher            = {IOP Publishing},
 title                = {Null space and resolution in dynamic computerized tomography},
 volume               = {32},
 year                 = {2016},
 }

@article{hahn2017a,
 author               = {B. Hahn},
 journal              = {Sensing and Imaging},
 pages                = {1-20},
 title                = {Motion estimation and compensation strategies in dynamic computerized tomography},
 volume               = {18},
 year                 = {2017},
 }

@article{ilmavirta2024tensor,
 author               = {Ilmavirta, Joonas and Kykk{\"a}nen, Antti},
 journal              = {The Journal of Geometric Analysis},
 number               = {147},
 publisher            = {Springer},
 title                = {Tensor tomography on negatively curved manifolds of low regularity},
 volume               = {34},
 year                 = {2024},
 }

@techreport{juhlin1992,
 address              = {SE-221 00 Lund, Schweden},
 author               = {P.~Juhlin},
 institution          = {Center for Mathematical Sciences, Lund Institute of Technology},
 title                = {Principles of {D}oppler tomography},
 year                 = {1992},
 }

@article{kaltenbacher2017,
 author               = {Kaltenbacher, Barbara},
 journal              = {Inverse Problems},
 note                 = {ID 064002},
 number               = {6},
 publisher            = {IOP Publishing},
 title                = {All-at-once versus reduced iterative methods for time dependent inverse problems},
 volume               = {33},
 year                 = {2017},
 }

@article{karassiov2004,
 author               = {V.P.~Karassiov and A.V.~Masalov},
 journal              = {Journal of Experimental and Theoretical Physics},
 number               = {1},
 pages                = {51-60},
 title                = {The method of polarization tomography of radiation in quantum optics},
 volume               = {99},
 year                 = {2004},
 }

@article{katsevich2010,
 author               = {Katsevich, A},
 journal              = {Inverse Problems},
 note                 = {ID 065007},
 number               = {6},
 publisher            = {IOP Publishing},
 title                = {An accurate approximate algorithm for motion compensation in two-dimensional tomography},
 volume               = {26},
 year                 = {2010},
 }

@article{lionheart2015,
 author               = {W.R.B.~Lionheart and P.J.~Withers},
 journal              = {Inverse Problems},
 note                 = {Article ID 045005},
 number               = {4},
 title                = {Diffraction tomography of strain},
 volume               = {31},
 year                 = {2015},
 }

@article{louis2002appl,
 author               = {U. Schmitt and A.K. Louis and C. Wolters and M. Vauhkonen},
 journal              = {Inverse Problems},
 pages                = {659-676},
 title                = {Efficient algorithms for the regularization of dynamic inverse problems: {I}{I}. {A}pplications},
 volume               = {18},
 year                 = {2002},
 }

@article{louis2002theo,
 author               = {U. Schmitt and A.K. Louis},
 journal              = {Inverse Problems},
 pages                = {645-658},
 title                = {Efficient algorithms for the regularization of dynamic inverse problems: {I}. {T}heory},
 volume               = {18},
 year                 = {2002},
 }

@article{louis2024unified,
 author               = {Louis, Alfred K},
 journal              = {Inverse Problems},
 note                 = {ID 085007},
 number               = {8},
 publisher            = {IOP Publishing},
 title                = {A unified approach to inversion formulae for vector and tensor ray and {R}adon transforms and the {N}atterer inequality},
 volume               = {40},
 year                 = {2024},
 }

@article{MONARD:14,
 author               = {F.~Monard},
 journal              = {SIAM J. Imag. Sci.},
 number               = {2},
 pages                = {1335-1357},
 title                = {Numerical implementation of geodesic {X}-ray transforms and their inversion},
 volume               = {7},
 year                 = {2014}
 }

@book{natterer_book,
 author               = {Natterer, Frank},
 publisher            = {Wiley},
 title                = {The Mathematics of Computerized Tomography},
 year                 = {1986},
 }

@article{neubauer2017nesterov,
 author               = {Neubauer, Andreas},
 journal              = {Journal of Inverse and Ill-posed Problems},
 number               = {3},
 pages                = {381--390},
 publisher            = {De Gruyter},
 title                = {On {N}esterov acceleration for {L}andweber iteration of linear ill-posed problems},
 volume               = {25},
 year                 = {2017},
 }

@article{norton1988,
 author               = {S.J. Norton},
 journal              = {Geophysics Journal},
 pages                = {161-168},
 title                = {Tomographic reconstruction of 2-{D} vector fields: Application to flow imaging},
 volume               = {97},
 year                 = {1988},
 }

@article{panin2002,
 author               = {V.Y.~Panin and G.L.~Zeng and M.~Defrise and G.T.~Gullberg},
 journal              = {Phys. Med. Biol.},
 pages                = {2737-2757},
 title                = {{Diffusion tensor MR imaging of principal directions: a tensor tomography approach}},
 volume               = {47},
 year                 = {2002},
 }

@article{paternain2021sharp,
 author               = {Paternain, Gabriel P and Salo, Mikko},
 journal              = {Mathematische Zeitschrift},
 number               = {3},
 pages                = {1323--1344},
 publisher            = {Springer},
 title                = {A sharp stability estimate for tensor tomography in non-positive curvature},
 volume               = {298},
 year                 = {2021},
 }

@inproceedings{polyakovahahn2019,
 author               = {Polyakova, Anna P and Svetov, Ivan E and Hahn, Bernadette N},
 booktitle            = {International Conference on Numerical Computations: Theory and Algorithms},
 organization         = {Springer},
 pages                = {446--453},
 title                = {The Singular Value Decomposition of the Operators of the Dynamic Ray Transforms Acting on 2{D} Vector Fields},
 year                 = {2019},
 }

@article{PolyakovaSvetov2024,
 author               = {Anna P. Polyakova and Ivan E. Svetov},
 doi                  = {doi:10.1515/jiip-2022-0019},
 journal              = {Journal of Inverse and Ill-posed Problems},
 number               = {1},
 pages                = {145--160},
 title                = {A numerical solution of the dynamic vector tomography problem using the truncated singular value decomposition method},
 url                  = {https://doi.org/10.1515/jiip-2022-0019},
 volume               = {32},
 year                 = {2024},
 }

@article{prince1996convolution,
 author               = {Prince, Jerry L},
 journal              = {IEEE Transactions on Image Processing},
 number               = {10},
 pages                = {1462--1472},
 publisher            = {IEEE},
 title                = {Convolution backprojection formulas for 3-{D} vector tomography with application to {M}{R}{I}},
 volume               = {5},
 year                 = {1996},
 }

@article{puro2007,
 author               = {A.\'{E}.~Puro and D.D.~Karov},
 journal              = {Optics and Spectroscopy},
 number               = {4},
 pages                = {678-682},
 title                = {Tensor field tomography of residual stresses},
 volume               = {103},
 year                 = {2007},
 }

@article{puro2016,
 author               = {A.\'{E}.~Puro and D.D.~Karov},
 journal              = {Journal of Thermal Stresses},
 number               = {5},
 pages                = {500-512},
 title                = {Inverse problem of thermoelasticity of fiber gratings},
 volume               = {39},
 year                 = {2016},
 }

@article{schuster2008,
 author               = {Schuster, Thomas},
 journal              = {In: Mathematical Methods in Biomedical Imaging and Intensity-Modulated Radiation Therapy (IMRT)},
 note                 = {Scuola Normale Superiore Pisa},
 pages                = {389--424},
 title                = {20 years of imaging in vector field tomography: a review},
 year                 = {2008},
 }

@article{sharafutdinov1992,
 author               = {Sharafutdinov, VA},
 journal              = {Siberian Mathematical Journal},
 number               = {3},
 pages                = {524--533},
 publisher            = {Springer},
 title                = {Integral geometry of a tensor field on a manifold whose curvature is bounded above: To {M}ikhail {M}ikhailovich {L}avrent'ev on his sixtieth birthday},
 volume               = {33},
 year                 = {1992},
 }

@article{SHARAFUTDINOV:2007,
 author               = {V.A.~Sharafutdinov},
 journal              = {Inverse Problems},
 pages                = {2603--2627},
 title                = {Slice-by-slice reconstruction algorithm for vector tomography with incomplete data},
 volume               = {23},
 year                 = {2007},
 }

@book{Sharafutdinov_1994,
 author               = {Sharafutdinov, V.A.},
 location             = {Zeist, The Netherlands},
 publisher            = {De Gruyter},
 title                = {Integral Geometry of Tensor Fields},
 year                 = {1994},
 }

@article{stefanov2018inverting,
 author               = {Stefanov, P. and Uhlmann, G. and Vasy, A.},
 journal              = {Journal d'Analyse Mathematique},
 number               = {1},
 pages                = {151--208},
 publisher            = {Springer},
 title                = {Inverting the local geodesic {X}-ray transform on tensors},
 volume               = {136},
 year                 = {2018},
 }

@article{STEFANOV;UHLMANN:04,
 author               = {P.~Stefanov and G.~Uhlmann},
 journal              = {Duke Math. J.},
 number               = {3},
 pages                = {445-467},
 title                = {Stability estimates for the {X}-ray transform of tensor fields and boundary rigidity},
 volume               = {123},
 year                 = {2004},
 }

@article{strahlen1998,
 author               = {Sparr, Gunnar and Strahlen, K},
 journal              = {IMA Volumes in Mathematics and its Applications; Computational Radiology and Imaging: Therapy and Diagnostic},
 publisher            = {Citeseer},
 title                = {Vector field tomography: an overview},
 volume               = {110},
 year                 = {1998},
 }

@article{udoschuster,
 author               = {U.~Schr{\"o}der and T.~Schuster},
 journal              = {Inverse Problems},
 note                 = {ID 085009},
 number               = {8},
 publisher            = {IOP Publishing},
 title                = {An iterative method to reconstruct the refractive index of a medium from time-of-flight measurements},
 volume               = {32},
 year                 = {2016},
 }

@article{vierus,
 author               = {Vierus, L. and Schuster, T.},
 journal              = {Electronic Transactions on Nuemrical Analysis},
 pages                = {80-100},
 title                = {Well-defined forward operators in dynamic diffractive tensor tomography using viscosity solutions of transport equations},
 volume               = {57},
 year                 = {2022},
 }

@book{webb2022introduction,
 author               = {Webb, Andrew},
 publisher            = {John Wiley \& Sons},
 title                = {Introduction to Biomedical Imaging},
 year                 = {2022},
 }

\end{document}